\documentclass[11pt]{article}
\ifdefined\pdfinfoomitdate\pdfinfoomitdate=1\fi
\ifdefined\pdftrailerid\pdftrailerid{}\fi

\usepackage[a4paper,margin=0.82in]{geometry}
\usepackage{amsmath,amssymb,amsthm,mathtools}
\usepackage{enumitem}
\usepackage{titlesec}

\titlespacing*{\section}{0pt}{1.6ex plus .4ex minus .2ex}{0.8ex}
\titlespacing*{\subsection}{0pt}{1.2ex plus .3ex minus .2ex}{0.6ex}

\setlist{nosep,leftmargin=*}
\usepackage[T1]{fontenc}
\usepackage{lmodern}
\usepackage{microtype}
\usepackage[hidelinks]{hyperref}
\hypersetup{
	pdftitle={Residue Sums Modulo Products of Two Distinct Odd Primes},
	pdfauthor={Haobo Ma, Mircea Merca, Wenlin Zhang},
	pdfsubject={Number Theory (math.NT)},
	pdfkeywords={residue sums, cyclic subgroups, Dirichlet characters, math.NT}
}

\newtheorem{theorem}{Theorem}[section]
\newtheorem{proposition}[theorem]{Proposition}
\newtheorem{lemma}[theorem]{Lemma}
\newtheorem{corollary}[theorem]{Corollary}
\theoremstyle{remark}
\newtheorem{remark}[theorem]{Remark}
\newtheorem{example}[theorem]{Example}

\newcommand{\ord}{\operatorname{ord}}
\newcommand{\eps}{\varepsilon}
\newcommand{\vp}{\nu_2}
\newcommand{\F}{\mathbb{F}}
\newcommand{\Z}{\mathbb{Z}}
\newcommand{\Ree}{\operatorname{Re}}
\newcommand{\lcm}{\operatorname{lcm}}
\newcommand{\Res}[2]{\langle #1\rangle_{#2}}

\title{Residue Sums Modulo Products of Two Distinct Odd Primes}
\author{
    Haobo Ma\thanks{ChronoAI Pte Ltd, Singapore, and The Omega Institute, auric@aelf.io}
    \and
    Mircea Merca\thanks{Corresponding author. Department of Mathematical Methods and Models, Fundamental Sciences Applied in Engineering Research Center, National University of Science and Technology Politehnica Bucharest, RO-060042 Bucharest and Academy of Romanian Scientists, RO-050044 Bucharest, Romania, mircea.merca@upb.ro}
    \and
    Wenlin Zhang\thanks{National University of Singapore (NUS), Singapore, and The Omega Institute, e1327962@u.nus.edu}
}
\date{}

\begin{document}
	
	\maketitle
	
	\begin{abstract}
Let $S_m(a)$ denote the sum of the least positive residues in one
multiplicative period of $a$ modulo $m$, of length $k$.
We study the equality $S_m(a)=mk/2$ for $m=pq$, where $p,q$ are
distinct odd primes and $k$ is even. Complementary pairing settles
subgroups containing $-1$; the remaining cases depend on their embedding
in the two local unit groups. At index four, we exclude mixed cyclic
subgroups with biquadratic annihilator and classify every eligible
imprimitive quartic kernel using two odd integer moments at its prime
conductor. For each fixed conductor, the proportion of admissible second
primes giving equality is $1/4$ or $1/2$; full coset balance has proportion
$1/4$. For primitive quartic characters, classical triangular
character-sum identities yield explicit short-sum criteria, and the
relative class-number formula gives an obstruction to a centered coset.
At index six, we give structural counts and construct an infinite family
modulo $19q$ in which nonzero sextic and quadratic moments cancel.
These applications use a standard odd-character expansion, which also
recovers the known index-two balance classification and determines the
sign of its deviation. A reproducible census for $pq\le10000$ counts
$75\,090$ eligible cyclic subgroups and identifies $101$ mixed index-four
subgroups whose own sum is centered although some coset sums are not.
\end{abstract}
	

	
	\section{Introduction}
	
	When does multiplicative symmetry alone make a sum of least positive
	residues exactly centered? For positive integers $a,m$ with $m>1$ and $(a,m)=1$,
	put $k=\ord_m(a)$ and
	\[
	S_m(a)=\sum_{i=1}^{k}(a^i\bmod m),
	\]
	where each residue belongs to $\{1,\dots,m-1\}$. We study the identity
	\begin{equation}\label{eq:main}
	S_m(a)=\frac{mk}{2}.
	\end{equation}
	The powers in this sum run once through the cyclic subgroup
	$H=\langle a\rangle$. Thus the question concerns the subgroup itself,
	and the identity is independent of its chosen generator. We call $H$
	\emph{sum-balanced} when \eqref{eq:main} holds.
	
	If $k$ is even and $-1\in H$, every residue is paired with its
	complement to $m$, and equality follows. For a prime modulus, even
	order forces this pairing. A composite modulus admits other elements
	of order two, so even order need not suffice. Merca
	\cite[Section~4, Conjecture~1]{Merca2011} conjectured \eqref{eq:main}
	under the additional hypotheses $(a-1,m)=1$ and $k$ even. The example
	$(a,m)=(2,15)$ gives
	\[
	k=4,\qquad S_{15}(2)=2+4+8+1=15\ne30.
	\]
This counterexample has a Lean~4 formalization. A separate formalization
proves Merca's lifted residue-sum Conjecture~2: for a prime $p$, a positive
integer $a$ coprime to $p$, and even $k=\ord_p(a)$,
\[
 \sum_{i=1}^{k}\bigl((2a^i+p)\bmod 2p\bigr)=pk.
\]
The latter is a lifted identity, distinct from \eqref{eq:main}.\footnote{
The source files for
\href{https://github.com/the-omega-institute/trureturing/blob/434cd5ce5a33609c7d2ecd2d450931460a8d1858/D5/S0/Certificates/MercaResidueSumEvenOrderRefutation.lean}{the counterexample}
and
\href{https://github.com/the-omega-institute/trureturing/blob/434cd5ce5a33609c7d2ecd2d450931460a8d1858/D5/S3/Arith/Congruence/MercaLiftedResidueSumPrimeOrder.lean}{the prime-modulus lifted identity}
are pinned to the source revision recorded in
\nolinkurl{data-package-v12/PROVENANCE.md}.  That provenance file gives
the exact declaration names and statements read at the immutable revision;
the package is prepared for the public repository
\url{https://github.com/the-omega-institute/residue-sums-pq-data}.
These two statements are the scope of the cited formalizations.}
A standalone census program and the complete subgroup data accompany
this manuscript as supplementary files.
	The structural results below use $(a,pq)=1$ and even order, and do not
	require $(a-1,pq)=1$.
	
	For $m=pq$, with $p,q$ distinct odd primes, the half-order element is
	$(-1,-1)$, $(-1,1)$ or $(1,-1)$ under the Chinese remainder theorem.
	The last two are the \emph{mixed} cases. They contain both sum-balanced
	and unbalanced subgroups, even with the same local orders and the same
	half-order element: modulo $85$, the subgroups generated by $2$ and
	$42$ have order $8$, local orders $(4,8)$, and half-order element $16$,
	but their sums are $170$ and $340$, respectively.

	\textbf{Relation to earlier work.}
	Pomerance and Ulmer \cite{PomeranceUlmer} call a subgroup
	\emph{balanced} if every one of its cosets has equally many
	representatives in the lower and upper halves of the modulus.
	Their Theorem~2.1 expresses this property through the separate
	vanishing of the half-interval sums of all odd characters trivial
	on the subgroup, while their Proposition~2.4 supplies a
	nonvanishing criterion for individual character sums. For odd
	moduli, their balance property is also equivalent to requiring
	a centered first moment on \emph{every} coset, as explained after
	Theorem~\ref{thm:character-criterion}. Our sum-balanced condition
	requires a centered first moment only on $H$ itself. Thus it is
	weaker: the odd-character moments may cancel without vanishing
	separately. The character expansion in
	Theorem~\ref{thm:character-criterion} follows from standard character
	orthogonality and serves as a framework for the arithmetic results
	below. At index two the two notions coincide in the mixed case;
	Theorem~\ref{thm:index2} specializes
	\cite[Theorem~2.1 and Proposition~2.4]{PomeranceUlmer}, while our
	class-number proof also determines the strict sign of the deviation by
    the classical imaginary quadratic class-number formula.
	
	Engberg \cite{Engberg2014} subsequently investigated the distribution,
	as the modulus varies, of balanced cyclic subgroups generated by a fixed
	integer, proving a conjecture of Pomerance and Ulmer about the role of
	two special families. His balance condition concerns all cosets and
	equality of their lower- and upper-half cardinalities. Our question is
	different: we restrict to moduli $pq$, vary the cyclic subgroup rather
	than fixing a generator, and require a centered first moment only on
	$H$ itself. Accordingly, our finite census does not assert an
	asymptotic frequency for fixed generators.
	
	\textbf{Contributions of this paper.}
	Our principal results concern the structure and arithmetic of
	sum-balanced \emph{cyclic} subgroups modulo a product of two distinct
	odd primes. More specifically:
	\begin{enumerate}[label=\textup{(\roman*)}]
		\item We describe the CRT linkage of cyclic subgroups with fixed
		local projections (Lemma~\ref{lem:linkage}) and show that the local
		orders and even the half-order element need not determine the
		residue sum, as the example modulo $85$ demonstrates.
		\item At index four, we prove that a mixed cyclic subgroup with a
		biquadratic annihilator always has a strict deficit
		(Proposition~\ref{prop:V4}), thereby excluding equality in this
		case.
        \item For imprimitive quartic annihilators, we classify all eligible
        cyclic kernels using two integer moments at the prime conductor
        (Theorem~\ref{thm:all-imprimitive-quartic}), and determine the
        fixed-conductor proportions of sum-balanced and coset-balanced
        lifts. The conductor-$13$ family gives explicit congruence classes
        (Corollary~\ref{thm:imprimitive13family}).
        \item For primitive quartic annihilators, we specialize the classical
        triangular Bernoulli identity to the residue deviation
        (Corollary~\ref{cor:primitive-quartic-divisibility}) and obtain
        short-sum criteria at $p=5$ and $p=13$ (Corollaries~\ref{cor:p5}
        and \ref{cor:p13}). The classical relative class-number formula
        gives a necessary square condition for a centered coset
        (Proposition~\ref{prop:class-number-obstruction}).
        \item At index six, we give an explicit structural specialization
        of rank-two subgroup counting, including the exact split according
        to the diagonal involution (Theorem~\ref{thm:index6-structure}),
        and characterize the local sextic characters with cyclic kernel
        (Proposition~\ref{prop:index6-character-cyclicity}). An infinite
        family modulo $19q$ exhibits cancellation of nonzero quadratic
        and sextic moments (Theorem~\ref{thm:sextic19family}); the primitive
        example modulo $35$ gives a separate instance of this phenomenon.
		\item We carry out a reproducible census of cyclic subgroups for
		$pq\le 10000$, classified by index and local character type.
		In particular, $101$ mixed sum-balanced subgroups of index four
		fail the stronger Pomerance--Ulmer balance condition; see
		Table~\ref{tab:coset-counts} and Example~\ref{ex:coset85} for a
		complete coset-sum check.
	\end{enumerate}
	
	The residue sum also has a direct interpretation in positional
	expansions. Let $a\ge2$, and let $c_1,\dots,c_k$ be one period of the
	base-$a$ expansion of $1/m$. If $r_0=1$ and
	$r_j=a^j\bmod m$, long division gives
	$ar_{j-1}=mc_j+r_j$, with $r_k=r_0$. Summing over the period yields
	\begin{equation}\label{eq:digit-link}
	c_1+\cdots+c_k=\frac{a-1}{m}S_m(a).
	\end{equation}
	Thus \eqref{eq:main} says that the mean digit is $(a-1)/2$.
	When $a^{k/2}\equiv-1\pmod m$, corresponding digits in the two
	half-periods add to $a-1$, the base-$a$ complementarity in Midy's
	theorem; see Lewittes \cite[Theorem~2(iv), Theorem~6 and its proof]{Lewittes}.
Girstmair's prime-denominator formulas \cite{Girstmair} connect digit
sums with class-number factors. Pujahari and Saikia \cite{PujahariSaikia}
study digit averages for products of two distinct primes and prime
powers, expressing them through $(a-1)/2$ and traces of generalized
Bernoulli numbers $B_{1,\chi}$ for odd characters. Thus the
composite-denominator Bernoulli framework of \eqref{eq:digit-link}
and Theorem~\ref{thm:character-criterion} has direct precedents.
Here we use this framework to decide vanishing and cancellation for
cyclic subgroups of specified small index, with the explicit arithmetic
classifications and families listed above.

Girstmair's broader treatment \cite{Girstmair1997} works with
$(\Z/m\Z)^\times/H$ for general $m$ and relates residue-distribution
vectors, their lengths and signs to class-number factors. His recent
survey \cite[Sections~2--3 and 6]{Girstmair2026} develops the geometry of
deviations for power residues, especially sixth-power residues.
Our class-number obstruction is an application of this classical
connection to centered first moments; the modulus being composite
alone does not distinguish that connection from the earlier work.

	The partial sums that arise in the quartic reductions belong to the
	classical study of the distribution of quadratic residues over
	subintervals. The published literature includes Berndt's study of classical theorems
	on quadratic residues \cite{Berndt} and Berndt and Chowla's work on
	zero sums of the Legendre symbol \cite{BerndtChowla}. We use these classical partial sums to test the specified quartic
    kernels; our $p=5$ criterion selects intervals ending at
	$q/5$ or $2q/5$ according to $q\bmod5$. The examples $q=17$ and $q=37$
	show that $q\bmod20$ alone does not determine the answer.
	
	The paper proceeds from the subgroup formulation and the general
	character expansion to the index-two classification, a single section
	on index four, and the index-six decomposition. We then give the
	reproducible census. Counting subgroups removes the weight
	$\varphi(|H|)$ attached to their generators, and separating the local
	quartic types identifies which reductions account for the observed
	equalities.
	
	\section{The subgroup formulation}
	
	For $x\in(\Z/m\Z)^\times$, let $\Res{x}{m}\in\{1,\dots,m-1\}$ denote
	its least positive representative.
	
	\begin{lemma}\label{lem:subgroup-sum}
		Let $H=\langle a\rangle$ and $k=|H|=\ord_m(a)$.  Then
		\[
		S_m(a)=\sum_{h\in H}\Res{h}{m}.
		\]
		In particular, $S_m(a)$ depends only on $H$ and not on the chosen
		generator of $H$.
	\end{lemma}
	
	\begin{proof}
		The residues $a,a^2,\dots,a^k$ run through the elements of $H$
		exactly once.
	\end{proof}
	
	\begin{proposition}[Complementary pairing]\label{prop:minus-one}
		Assume that $k$ is even.  If $-1\in H$, equivalently
		\(a^{k/2}\equiv-1\pmod m\),
		then
		\[
		S_m(a)=\frac{mk}{2}.
		\]
	\end{proposition}
	
	\begin{proof}
		The involution $h\mapsto-h$ has no fixed points on $H$.  Moreover,
		\[
		\Res{h}{m}+\Res{-h}{m}=m.
		\]
		Hence the $k$ elements of $H$ split into $k/2$ complementary pairs.
	\end{proof}
	
	We now specialize to $m=pq$.  Put
	\(u=\ord_p(a)\), $v=\ord_q(a)$, $k=\lcm(u,v)$.

	\begin{proposition}[The half-order sign]\label{prop:2adic}
		Assume that $k$ is even and put
		$\alpha=\vp(u)$, $\beta=\vp(v)$.
		Then:
		\begin{enumerate}[label=\textup{(\roman*)}]
			\item if $\alpha=\beta\ge1$, then
			\[
			a^{k/2}\equiv(-1,-1)\pmod{(p,q)};
			\]
			\item if $\alpha>\beta$, then
			\[
			a^{k/2}\equiv(-1,1)\pmod{(p,q)};
			\]
			\item if $\beta>\alpha$, then
			\[
			a^{k/2}\equiv(1,-1)\pmod{(p,q)}.
			\]
		\end{enumerate}
		Thus the mixed cases occur exactly when
		\[
		\vp(\ord_p(a))\ne\vp(\ord_q(a)).
		\]
	\end{proposition}
	
	\begin{proof}
		For the $p$-component, $a^{k/2}\equiv1\pmod p$ if and only if
		$u\mid k/2$.  Since $\vp(k)=\max\{\alpha,\beta\}$, this occurs exactly
		when $\alpha<\vp(k)$.  Otherwise $a^{k/2}$ is the unique element of
		order $2$ in $\langle a\bmod p\rangle$, hence is $-1$.  The argument
		for the $q$-component is identical.
	\end{proof}

	\begin{remark}
		The local orders and the half-order element do not determine whether
		$H$ is sum-balanced. For $m=85=5\cdot17$, the bases $2$ and $42$ satisfy
		\[
		\ord_5(2)=\ord_5(42)=4,\qquad
		\ord_{17}(2)=\ord_{17}(42)=8.
		\]
		Both subgroups therefore have order $k=8$, and
		\[
		2^4\equiv42^4\equiv16\pmod{85},\qquad
		16\equiv(1,-1)\pmod{(5,17)}.
		\]
		Their residue cycles and sums are
		\[
		\begin{array}{c|l|r}
		a&\text{residue cycle}&S_{85}(a)\\ \hline
		2&2,4,8,16,32,64,43,1&170\\
		42&42,64,53,16,77,4,83,1&340.
		\end{array}
		\]
		Since $85\cdot8/2=340$, the first subgroup is unbalanced and the
		second is sum-balanced. Thus the embedding of $H$ into the two local
		unit groups is essential, even with the local orders and mixed sign
		pattern fixed.
	\end{remark}

	Let
	\[
	G=(\Z/pq\Z)^\times,\qquad d=[G:H]=\frac{(p-1)(q-1)}{k}.
	\]
	
	\begin{proposition}\label{prop:index-divisibility}
		One has
		\[
		\gcd(p-1,q-1)\mid d.
		\]
		In particular, $d$ is always even.
	\end{proposition}
	
	\begin{proof}
		Since
		\[
		k=\lcm(u,v)\mid\lcm(p-1,q-1)
		=\frac{(p-1)(q-1)}{\gcd(p-1,q-1)},
		\]
		the assertion follows after dividing $(p-1)(q-1)$ by $k$.  Since
		$p-1$ and $q-1$ are even, the final statement is immediate.
	\end{proof}
	
	\begin{remark}
		There are no cyclic subgroups of odd index arising as $\langle a\rangle$
		in this problem.  In particular, after index $2$ the next possible
		index is $4$.
	\end{remark}
	
	The following parametrization is the cyclic case of Goursat's
    lemma; compare \cite[Theorem~4.4(a)]{BauerSenZvengrowski}. The
    general rank-two cyclic-subgroup count used in the later census is
    the specialization of Tóth~\cite[Theorems~4.6--4.7]{Toth2014}.
    \begin{lemma}[CRT linkage]\label{lem:linkage}
		Let $A=\langle b\rangle$ and $B=\langle c\rangle$ be cyclic groups of
		orders $u$ and $v$, and put $g=\gcd(u,v)$. The cyclic subgroups of
		$A\times B$ that project onto both factors are exactly
		\[
		H_\lambda
		=\{(b^i,c^j):j\equiv\lambda i\pmod g\},
		\qquad
		\lambda\in(\Z/g\Z)^\times.
		\]
		Each has order $\lcm(u,v)$, and there are $\varphi(g)$ such subgroups.
		For $g=1$ there is one subgroup, namely $A\times B$.
	\end{lemma}
	
	\begin{proof}
		A generator projecting onto both factors has the form $(b^s,c^t)$,
		where $(s,u)=(t,v)=1$. The pair $(b^i,c^j)$ belongs to its cyclic
		subgroup precisely when
		\[
		sn\equiv i\pmod u,\qquad tn\equiv j\pmod v
		\]
		have a common solution. By the generalized Chinese remainder theorem,
		this is equivalent to
		\[
		s^{-1}i\equiv t^{-1}j\pmod g,
		\]
		or $j\equiv ts^{-1}i\pmod g$. Conversely, a unit $\lambda$ modulo
		$g$ lifts to a unit $t$ modulo $v$: choose nonzero residues at the
		primes dividing $v$ but not $g$, and apply the Chinese remainder
		theorem. Then $(b,c^t)$ generates $H_\lambda$ and has order
		$\lcm(u,v)$. Distinct parameters give distinct permitted pairs
		with first exponent $i=1$.
	\end{proof}
	
	In the example modulo $85$, take
	$A=\langle2\bmod5\rangle$ and $B=\langle2\bmod17\rangle$.
	Then $u=4$, $v=8$ and $g=4$, and the two subgroups correspond to
	$\lambda=1$ and $\lambda=3$, respectively, since
	$42\equiv2\pmod5$ and $42\equiv2^3\pmod{17}$.
	Both subgroups have index $8$.
	
	\section{An exact character criterion}
	
	Let $\widehat G$ be the character group of $G$, and let
	\[
	H^\perp=\{\chi\in\widehat G:\chi(h)=1\text{ for all }h\in H\}
	\]
	be the annihilator of $H$.  Thus $|H^\perp|=d$.  We extend each
	character to a Dirichlet character modulo $pq$ by setting it equal to
	$0$ on nonunits, and define
	\[
	A_{pq}(\chi)
	=
	\sum_{\substack{1\le n<pq\\(n,pq)=1}} n\chi(n).
	\]

	We shall also use the standard relation between the first moment of a
	nonprincipal Dirichlet character and its Dirichlet \(L\)-function
	(see, for example, \cite[Chapter~4, pp.~30--46]{Washington}).
	If \(\chi\) is a nonprincipal Dirichlet character modulo \(N\), then
	\[
	L(0,\chi)
	=
	-\frac{1}{N}\sum_{n=1}^{N} n\chi(n),
	\]
	and consequently
	\[
	A_N(\chi):=\sum_{n=1}^{N}n\chi(n)=-N\,L(0,\chi).
	\]

	\begin{lemma}\label{lem:even-char}
		If $\chi$ is nonprincipal and even, then
		\[
		A_{pq}(\chi)=0.
		\]
		For the principal character $\chi_0$,
		\[
		A_{pq}(\chi_0)=\frac{pq\,\varphi(pq)}{2}.
		\]
	\end{lemma}
	
	\begin{proof}
		For an even nonprincipal character, pair $n$ with $pq-n$:
		\[
		n\chi(n)+(pq-n)\chi(pq-n)
		=pq\chi(n).
		\]
		Summing over one element from each pair gives zero because
		$\sum_{n\in G}\chi(n)=0$.  The principal case follows from the same
		pairing without the character.
	\end{proof}
	
	\begin{theorem}[Character criterion]\label{thm:character-criterion}
		Let $H=\langle a\rangle$ have index $d$ in $G$.  Then
		\begin{equation}\label{eq:character-criterion}
		S_{pq}(a)-\frac{pq\,k}{2}
		=
		\frac1d
		\sum_{\substack{\chi\in H^\perp\\ \chi(-1)=-1}}
		A_{pq}(\chi).
		\end{equation}
		Consequently,
		\[
		S_{pq}(a)=\frac{pq\,k}{2}
		\]
		if and only if the sum on the right-hand side of
		\eqref{eq:character-criterion} vanishes.
	\end{theorem}
	
	\begin{proof}
		Character orthogonality gives
		\[
		1_H(x)=\frac1d\sum_{\chi\in H^\perp}\chi(x).
		\]
		Hence
		\[
		S_{pq}(a)
		=
		\frac1d\sum_{\chi\in H^\perp}A_{pq}(\chi).
		\]
		The principal contribution is
		\[
		\frac1d\cdot\frac{pq\,\varphi(pq)}2
		=\frac{pq\,k}{2},
		\]
		and all nonprincipal even characters vanish by
		Lemma~\ref{lem:even-char}.
	\end{proof}
	
	For comparison with \cite[Theorem~2.1]{PomeranceUlmer}, let $N$ be
	odd, let $\chi$ be a nonprincipal character modulo $N$, and put
	$c_\chi=\sum_{1\le n<N/2}\chi(n)$. Permuting the units by
	$n\mapsto[2n]_N$ gives
	\[
	A_N(\chi)=\chi(2)\bigl(2A_N(\chi)+Nc_\chi\bigr),
	\qquad
	c_\chi=\bigl(\chi(2)^{-1}-2\bigr)\frac{A_N(\chi)}N.
	\]
	The factor $\chi(2)^{-1}-2$ is nonzero. Consequently their
	half-interval coefficient vanishes precisely when the first moment
	vanishes. Applying character orthogonality to the centered sum over
	$gH$, for all $g$, shows that every coset is sum-balanced if and only
	if $A_N(\chi)=0$ for every odd $\chi\in H^\perp$.
	In contrast, \eqref{eq:character-criterion} asks for the sum of those
	moments to vanish. At index two there is one odd contribution in the
	mixed case, so the two balance notions coincide; at indices four and
	six the distinction permits cancellation.
	
	\section{Subgroups of index two}
	
	Assume that $d=2$ and that we are in a mixed case.  Then
	$H^\perp=\{1,\chi\}$,
	where $\chi$ is an odd quadratic character.  Write
	\[
	\eps_p(n)=\left(\frac np\right),\qquad
	\eps_q(n)=\left(\frac nq\right).
	\]
	
	In sums modulo $pq$, the symbols $\eps_p$ and $\eps_q$ denote the
	characters induced to modulus $pq$, extended by zero on all nonunits.
	In $A_p(\eps_p)$ and $A_q(\eps_q)$ they denote the primitive local
	characters at the displayed prime modulus.
	
	We shall use the imaginary quadratic class-number formula
	(see \cite[Chapter~4, pp.~30--46]{Washington}).  If $D$ is a negative fundamental
	discriminant, $\chi_D$ is its primitive quadratic character, $h(D)$
	is the class number and $w(D)$ is the number of roots of unity in
	$\mathbb Q(\sqrt D)$, then
	\begin{equation}\label{eq:quadratic-moment}
	L(0,\chi_D)=\frac{2h(D)}{w(D)}>0,
	\qquad
	A_{|D|}(\chi_D)=-\frac{2|D|h(D)}{w(D)}<0.
	\end{equation}
	The factor $w(D)$ includes the smallest conductor: $w(-3)=6$ and
	$A_3(\eps_3)=1-2=-1$.
	
	\begin{lemma}[Adjoining a prime to the modulus]\label{lem:induced-moment}
		Let $r$ and $\ell$ be distinct primes, let $\chi$ be a primitive
		nonprincipal Dirichlet character modulo $r$, and let $\widetilde\chi$
		be the character induced by $\chi$ modulo $r\ell$. Then
		\[
		L(s,\widetilde\chi)
		=\bigl(1-\chi(\ell)\ell^{-s}\bigr)L(s,\chi),
		\]
		and, with $A_N(\chi)=\sum_{n=1}^{N}n\chi(n)$ as above,
		\[
		A_{r\ell}(\widetilde\chi)
		=\ell\bigl(1-\chi(\ell)\bigr)A_r(\chi).
		\]
	\end{lemma}
	
	\begin{proof}
		For $\Ree s>1$, the induced character agrees with $\chi$ away from
		multiples of $\ell$ and vanishes on those multiples. Hence
		\[
		\begin{aligned}
		L(s,\widetilde\chi)
		&=L(s,\chi)-\sum_{j=1}^{\infty}\frac{\chi(\ell j)}{(\ell j)^s}\\
		&=\bigl(1-\chi(\ell)\ell^{-s}\bigr)L(s,\chi).
		\end{aligned}
		\]
		Analytic continuation gives the identity at $s=0$. Both characters
		are nonprincipal, so the normalization $A_N(\chi)=-N L(0,\chi)$ gives
		\[
		\begin{aligned}
		A_{r\ell}(\widetilde\chi)
		&=-r\ell\bigl(1-\chi(\ell)\bigr)L(0,\chi)\\
		&=\ell\bigl(1-\chi(\ell)\bigr)A_r(\chi).
		\end{aligned}
		\]
	\end{proof}
	
	The equality classification below specializes
	\cite[Theorem~2.1 and Proposition~2.4]{PomeranceUlmer}.
	Their Proposition~2.4 gives vanishing precisely when the character is
	induced from one prime and takes value $1$ at the other; a primitive
	character of conductor $pq$ cannot give vanishing. We give a
	self-contained class-number proof, which also shows that every failure
	of equality is a strict deficit.
	
	\begin{samepage}
		\begin{theorem}[Complete index-two classification]\label{thm:index2}
			Suppose $[G:H]=2$ and $-1\notin H$.
			\begin{enumerate}[label=\textup{(\roman*)}]
				\item If $\chi=\eps_p$, then
				\[
				S_{pq}(a)=\frac{pq\,k}{2}
				\iff
				\left(\frac qp\right)=1.
				\]
				\item If $\chi=\eps_q$, then
				\[
				S_{pq}(a)=\frac{pq\,k}{2}
				\iff
				\left(\frac pq\right)=1.
				\]
				\item If $\chi=\eps_p\eps_q$, then
				\[
				S_{pq}(a)\ne\frac{pq\,k}{2}.
				\]
			\end{enumerate}
		\end{theorem}
	\end{samepage}
	
	\begin{proof}
		By Theorem~\ref{thm:character-criterion},
		\[
		S_{pq}(a)-\frac{pq\,k}{2}
		=\frac12A_{pq}(\chi).
		\]
		If $\chi=\eps_p$, then $\chi$ is induced from the primitive quadratic
		character modulo $p$, and Lemma~\ref{lem:induced-moment} gives
		\[
		A_{pq}(\eps_p)
		=
		q\left(1-\left(\frac qp\right)\right)A_p(\eps_p).
		\]
		Since $\eps_p$ is odd, its fundamental discriminant is $-p$, and
		\eqref{eq:quadratic-moment} gives $A_p(\eps_p)<0$, proving (i).
		Part (ii) is symmetric. In case (iii), the oddness of $\eps_p\eps_q$
		means that exactly one of $p,q$ is congruent to $3$ modulo $4$.
		Hence $pq\equiv3\pmod4$, and $-pq$ is a negative fundamental
		discriminant. The product $\eps_p\eps_q$ is its primitive quadratic
		character.
		Thus \eqref{eq:quadratic-moment} gives
		\[
		A_{pq}(\eps_p\eps_q)=-pq\,L(0,\eps_p\eps_q)<0,
		\]
		and in particular this moment is nonzero.
	\end{proof}
	
	\begin{example}
		For $(p,q,a)=(3,11,5)$ one has $k=10$ and
		\[
		S_{33}(5)=165=\frac{33\cdot10}{2}.
		\]
		The relevant index-two character is $\eps_{11}$, and
		\[
		\left(\frac3{11}\right)=1.
		\]
	\end{example}
	
	\begin{example}
		For $(p,q,a)=(3,5,2)$ one has $k=4$ and
		\[
		S_{15}(2)=15\ne30.
		\]
		Here the relevant quadratic character is the primitive product
		$\eps_3\eps_5$.
	\end{example}
	
	\section{Subgroups of index four}
	
	Assume from now on that $d=4$.  Then
	$H^\perp\simeq C_4$ or $H^\perp\simeq C_2\times C_2$.

	\subsection{The biquadratic annihilator}

	\begin{proposition}\label{prop:V4}
		Suppose
		$H^\perp\simeq C_2\times C_2$.
		Then
		\[
		H^\perp=\{1,\eps_p,\eps_q,\eps_p\eps_q\},
		\qquad
		H=G^2.
		\]
		Since \(H=\langle a\rangle\) is cyclic, necessarily
		\[
		\gcd\left(\frac{p-1}{2},\frac{q-1}{2}\right)=1,
		\]
		equivalently
		\[
		\gcd(p-1,q-1)=2.
		\]
		If \(k\) is even, then exactly one of \(p,q\) is congruent to \(1\)
		modulo \(4\), and in the mixed case
		\[
		S_{pq}(a)<\frac{pq\,k}{2}.
		\]
		In particular, equality is impossible.
	\end{proposition}

	\begin{proof}
		The first statement follows because the full group of quadratic
		characters modulo $pq$ is
		\[
		\{1,\eps_p,\eps_q,\eps_p\eps_q\}.
		\]
		Thus \(H\) is the subgroup of squares. Under CRT,
		\[
		G^2\simeq C_{(p-1)/2}\times C_{(q-1)/2}.
		\]
		Since \(H=G^2\) is cyclic, the two factors must have coprime orders,
		and hence
		\[
		\gcd\left(\frac{p-1}{2},\frac{q-1}{2}\right)=1.
		\]
		Assume, without loss of generality, that
		\[
		p\equiv1\pmod4,\qquad q\equiv3\pmod4.
		\]
		The odd characters in $H^\perp$ are then
		\[
		\eps_q,\qquad \eps_p\eps_q.
		\]
		The character $\eps_q$ is induced from the odd primitive quadratic
		character modulo $q$.  By \eqref{eq:quadratic-moment}, with $D=-q$,
		$A_q(\eps_q)<0$, so Lemma~\ref{lem:induced-moment} gives
		\[
		A_{pq}(\eps_q)
		=
		p\left(1-\left(\frac pq\right)\right)A_q(\eps_q)\le0.
		\]
		Moreover, $\eps_p\eps_q$ is the primitive quadratic character of
		fundamental discriminant $-pq$, so \eqref{eq:quadratic-moment} gives
		\[
		A_{pq}(\eps_p\eps_q)<0.
		\]
		Theorem~\ref{thm:character-criterion} now gives the strict inequality.
	\end{proof}
	
	\begin{example}
		For $(p,q,a)=(5,7,4)$,
		\[
		k=6,\qquad
		S_{35}(4)=4+16+29+11+9+1=70,
		\]
		whereas
		\[
		\frac{35\cdot6}{2}=105.
		\]
	\end{example}
	
	The preceding proposition completely excludes the mixed sum-balanced case
	when the annihilator is biquadratic.  Thus, at index \(4\), the only
	remaining possibility is the cyclic case
	\[
	H^\perp\simeq C_4.
	\]
	This case is genuinely different: the quadratic characters alone no
	longer determine the residue sum, and the contribution of a quartic
	character must be taken into account.
	
	\subsection{The cyclic annihilator}
	
	Suppose now that
	$H^\perp=\langle\psi\rangle\simeq C_4$.
	In a mixed case $-1\notin H$.  Since the image of $-1$ in $G/H$ is
	the unique element of order $2$, one has
	$\psi(-1)=-1$.
	Thus
	$H^\perp=\{1,\psi,\psi^2,\overline{\psi}\}$,
	with $\psi$ and $\overline{\psi}$ odd and $\psi^2$ even.

	The exhaustive computation for \(pq\le 10000\) shows that balanced
	mixed examples do occur at index \(4\).  Since the biquadratic case is
	excluded by Proposition~\ref{prop:V4}, all such examples necessarily
	belong to the cyclic annihilator case.

	\begin{theorem}\label{thm:C4basic}
		In the cyclic-annihilator case,
		\begin{equation}\label{eq:C4deviation}
		S_{pq}(a)-\frac{pq\,k}{2}
		=
		\frac12\Ree A_{pq}(\psi).
		\end{equation}
		Hence
		\[
		S_{pq}(a)=\frac{pq\,k}{2}
		\iff
		\Ree A_{pq}(\psi)=0
		\iff
		\Ree L(0,\psi)=0.
		\]
	\end{theorem}
	
	\begin{proof}
		By Theorem~\ref{thm:character-criterion},
		\[
		S_{pq}(a)-\frac{pq\,k}{2}
		=
		\frac14\bigl(A_{pq}(\psi)+A_{pq}(\overline\psi)\bigr).
		\]
		Since
		$A_{pq}(\overline\psi)=\overline{A_{pq}(\psi)}$,
		the first assertion follows.  The second uses
		$A_{pq}(\psi)=-pq\,L(0,\psi)$.
	\end{proof}
	
	\subsection{The imprimitive quartic case}
	
\begin{lemma}[Lifting coset sums]\label{lem:coset-lifting}
Let $r,q$ be distinct odd primes, $K\le\F_r^\times$, and $C$ a coset
of $K$. Put $s(C)=\sum_{x\in C}\Res{x}{r}$ and identify
$\widetilde C=C\times\F_q^\times$ by the Chinese remainder theorem.
Then
\begin{equation}\label{eq:coset-lifting}
 \sum_{z\in\widetilde C}\Res{z}{rq}
 -\frac{rq\,|K|(q-1)}2
 =q\bigl(s(C)-s(q^{-1}C)\bigr).
\end{equation}
\end{lemma}
\begin{proof}
For each representative $x\in C$, its $q$ lifts are
$x,x+r,\ldots,x+(q-1)r$. Their total over $C$ is
$q\,s(C)+r|K|q(q-1)/2$. The lifts divisible by $q$ are precisely
$qy$ with $y\in\{1,\ldots,r-1\}$ representing $q^{-1}C$.
Removing their sum $q\,s(q^{-1}C)$ proves the identity.
\end{proof}

For an imprimitive quartic character induced from $\chi$ modulo $r$,
its kernel is $K\times\F_q^\times$, where $K=\ker\chi$ has order
$t=(r-1)/4$. This kernel is cyclic precisely when $\gcd(t,q-1)=1$.
Since $q-1$ is even, cyclicity forces $t$ odd, hence $r\equiv5\pmod8$.
The eligibility condition that both local projection orders exceed
one excludes $r=5$. The next theorem therefore covers every eligible
cyclic imprimitive quartic kernel, after interchanging the primes
if necessary.

\begin{theorem}\label{thm:all-imprimitive-quartic}
Let $r>5$ be prime with $r\equiv5\pmod8$, put $t=(r-1)/4$, and
choose a primitive root $g$ modulo $r$ and the quartic character
$\chi(g)=i$. Write $K=\ker\chi$, and define
\[
 u=\frac{2s(K)}r-t,\qquad v=\frac{2s(gK)}r-t.
\]
Then $u,v$ are nonzero odd integers. For every odd prime $q\ne r$
with $\gcd(t,q-1)=1$, the subgroup $H=K\times\F_q^\times$ is
cyclic, omits $-1$, has index four and order $t(q-1)$, and every
generator $a$ satisfies $(a-1,rq)=1$. Its sum is centered exactly when
\begin{equation}\label{eq:all-imprimitive-quartic}
 \chi(q)=1,\quad\text{or}\quad
 \bigl(u=v\ \text{and}\ \chi(q)=-i\bigr),\quad\text{or}\quad
 \bigl(u=-v\ \text{and}\ \chi(q)=i\bigr).
\end{equation}
Every coset of $H$ is centered if and only if $\chi(q)=1$.
For fixed $r$, among the primes $q$ satisfying $\gcd(t,q-1)=1$,
the limiting proportion giving a centered subgroup is
\[
 \begin{cases}
  1/4,&|u|\ne|v|,\\
  1/2,&|u|=|v|,
 \end{cases}
\]
and the limiting proportion giving full coset balance is $1/4$.
\end{theorem}
\begin{proof}
The direct-product criterion proves cyclicity. Both projection
orders of a generator, $t$ and $q-1$, exceed one. Since $t$ is odd,
$-1\notin K$, so $-1\notin H$.
For any coset $C$ of the nontrivial subgroup $K$, its elements sum
to zero in $\F_r$: multiplication by any $b\in K\setminus\{1\}$
fixes that sum. Hence $r\mid s(C)$, and $u,v$ are odd integers.
In particular, neither is zero. The four local coset sums are
\[
 s(K),\quad s(gK),\quad rt-s(K),\quad rt-s(gK),
\]
since $-1\in g^2K$. Thus $A_r(\chi)=r(u+iv)$.
Lemma~\ref{lem:induced-moment} and Theorem~\ref{thm:C4basic} give
\[
 S_{rq}(a)-\frac{rq\,t(q-1)}2
 =\frac{rq}{2}\Ree\bigl((1-\chi(q))(u+iv)\bigr).
\]
For $\chi(q)=1,-1,i,-i$, the real factors are respectively
$0,2u,u+v,u-v$. Since $u\ne0$, this proves
\eqref{eq:all-imprimitive-quartic}; in particular the $\chi(q)=-1$
branch never gives equality for an eligible cyclic kernel.
All coset deviations vanish exactly when the induced odd moment
$q(1-\chi(q))r(u+iv)$ vanishes, by character orthogonality.
Since $u+iv\ne0$, this is equivalent to $\chi(q)=1$.

For the proportions, let $R=\prod_{\ell\mid t}\ell$ be the product
of the distinct prime divisors of $t$. Apart from finitely many
primes, admissibility means $q\not\equiv0,1\pmod\ell$ for each
$\ell\mid t$. These conditions are independent of $q\bmod r$,
because $(r,R)=1$. Each quartic value occurs on exactly $(r-1)/4$
nonzero residue classes modulo $r$. The prime number theorem in
arithmetic progressions modulo $rR$ therefore gives equal relative
proportion $1/4$ to the four quartic values among admissible primes.
There is at least one admissible reduced class since every
$\ell\mid t$ is odd. Exactly one value is balanced when $|u|\ne|v|$,
and exactly two when $|u|=|v|$; full coset balance selects one.
\end{proof}

The proportions fix the conductor $r$ and vary the second prime;
they are distinct from a density for fixed generators or for the
bounded two-prime census. The remaining local question is to determine
which primes $r$ satisfy $|u|=|v|$.

	\begin{example}
		Take $(r,\ell,a)=(13,3,29)$.  Then
		\[
		k=\ord_{39}(29)=6
		\]
		and
		\[
		S_{39}(29)=29+22+14+16+35+1=117
		=\frac{39\cdot6}{2}.
		\]
		If the quartic character modulo $13$ is normalized by $\chi(2)=i$,
		then $3\equiv2^4\pmod{13}$, so $\chi(3)=1$, and
		Lemma~\ref{lem:induced-moment} gives $A_{39}(\widetilde\chi)=0$.
	\end{example}

	\begin{corollary}[The conductor-$13$ family]
		\label{thm:imprimitive13family}
		Let $q>3$ be a prime distinct from $13$ with $q\equiv2\pmod3$.
		Let $\chi_{13}$ be the primitive quartic character modulo $13$
		normalized by $\chi_{13}(2)=i$, let $\widetilde\chi_{13}$ be its
		extension by induction to modulus $13q$, and put
		$H=\ker\widetilde\chi_{13}$. Then $H$ is a cyclic subgroup of index
		four, $-1\notin H$, and $k=|H|=3(q-1)$. Every generator $a$ of
		$H$ satisfies $(a-1,13q)=1$, and
		\begin{equation}\label{eq:imprimitive13classification}
		S_{13q}(a)-\frac{13qk}{2}
		=\begin{cases}
		0,&q\bmod13\in\{1,3,7,8,9,11\},\\
		-13q,&q\bmod13\in\{2,4,5,6,10,12\}.
		\end{cases}
		\end{equation}
		In the equality case, all cosets of $H$ are sum-balanced if and only
		if $q\bmod13\in\{1,3,9\}$. If instead
		$q\bmod13\in\{7,8,11\}$, the subgroup $H$ is sum-balanced but
		\begin{equation}\label{eq:imprimitive13coset}
		\sum_{h\in 2H}\Res{h}{13q}=\frac{13qk}{2}-13q.
		\end{equation}
		Consequently, there are infinitely many cyclic subgroups of index
		four which are sum-balanced but not balanced in the sense of
		Pomerance and Ulmer; they arise, for instance, for all primes
		$q\equiv20\pmod{39}$.
	\end{corollary}
	
\begin{proof}
Here $K=\{1,3,9\}$, and the four local cosets on which $\chi_{13}$
takes the values $1,i,-1,-i$ are respectively
\[
 \{1,3,9\},\quad\{2,5,6\},\quad\{4,10,12\},\quad\{7,8,11\}.
\]
Their sums are $13,13,26,26$, so $u=v=-1$ in
Theorem~\ref{thm:all-imprimitive-quartic}. Its deviation formula
proves the classification. For $\chi_{13}(q)=-i$, the induced moment
is $-26iq$, and the coset $2H$ has deviation
$\frac12\Ree((-i)(-26iq))=-13q$.
The progression $q\equiv20\pmod{39}$ is reduced and has
$q\equiv2\pmod3$ and $q\equiv7\pmod{13}$, so Dirichlet's theorem
proves infinitude.
\end{proof}
	
	\medskip
	\noindent\textit{A complementary progression with full coset balance.}
	For every prime $q\equiv14\pmod{39}$, one has
	$q\equiv2\pmod3$ and $q\equiv1\pmod{13}$. Hence
	$\chi_{13}(q)=1$, and Lemma~\ref{lem:induced-moment} gives
	$A_{13q}(\widetilde\chi_{13})=0$. It follows that not only
	$H=\ker\widetilde\chi_{13}$ but every one of its cosets is
	sum-balanced. Since $\gcd(14,39)=1$, Dirichlet's theorem yields
	infinitely many primes in this progression. For the smallest one,
	$q=53$, the element $a=3$ generates $H$ modulo $689$, with
	$k=156$, and
	\[
	S_{689}(3)=53742=\frac{689\cdot156}{2}.
	\]
	This contrasts with the progression $q\equiv20\pmod{39}$,
	where $H$ is sum-balanced but not all its cosets are balanced.
	
	\begin{example}\label{ex:imprimitive13family}
		For $q=59$ and $a=42$, the element $a$ generates
		$H=\ker\widetilde\chi_{13}$ modulo $767$ and has order $k=174$.
		Direct summation yields
		\[
		S_{767}(42)=66729=\frac{767\cdot174}{2},
		\qquad
		\sum_{h\in2H}\Res{h}{767}=65962=66729-767.
		\]
		Thus the subgroup is sum-balanced, but its coset $2H$ is not.
	\end{example}

    Two smaller witnesses make the congruence classes concrete. For
    $q=5$ one may take $a=3$ modulo $65$; the four coset sums are
    $(325,390,455,390)$, so $H$ is below the center. For $q=11$,
    $N=143$ and $a=29$, the sums are
    $(2145,2002,2145,2288)=(2145,2145-143,2145,2145+143)$.
    These values agree with the sign table above and are independent
    checks of the noncentered cases.
	
	\subsection{CRT reduction for primitive quartic characters}
	
	Suppose that $\psi$ is primitive modulo $pq$.  Under CRT it can be
	written
	$\psi=\chi_p\chi_q$,
	where both local characters are nonprincipal and their orders are one
	of
	$(4,2)$, $(2,4)$, $(4,4)$.
	
	The local type \((2,4)\) is obtained from the type \((4,2)\) by
	interchanging \(p\) and \(q\).  Thus \((4,2)\) represents the first
	two types; the type \((4,4)\) is treated in
	Section~\ref{sec:primitive44}.
	
	For a nonprincipal character $\chi$ modulo a prime $q$, define
	\[
	C_\chi(t)=\sum_{1\le n<t}\chi(n),
	\qquad 0\le t\le q.
	\]
	
	\begin{lemma}[CRT reduction]\label{lem:CRT}
		Let $\chi_p$ and $\chi_q$ be nonprincipal characters modulo the
		distinct primes $p$ and $q$, respectively.  Write $[z]_q$ for the
		least nonnegative residue of $z$ modulo $q$, and put
		$d_x=[xp^{-1}]_q\in\{0,1,\dots,q-1\}$.
		Then
		\begin{equation}\label{eq:CRTformula}
		A_{pq}(\chi_p\chi_q)
		=
		pq\,\chi_q(p)
		\sum_{x=1}^{p-1}
		\chi_p(x)\,C_{\chi_q}(d_x).
		\end{equation}
	\end{lemma}
	
	\begin{proof}
		For $1\le x\le p-1$ and $1\le y\le q-1$, let $R(x,y)$ be the CRT
		representative in $\{1,\dots,pq-1\}$ with
		\[
		R(x,y)\equiv x\pmod p,\qquad
		R(x,y)\equiv y\pmod q.
		\]
		Write
		\[
		R(x,y)=x+pt_{x,y},
		\qquad
		t_{x,y}=[(y-x)p^{-1}]_q.
		\]
		The terms involving $x$ disappear after summation over $y$ because
		$\sum_y\chi_q(y)=0$.  Hence
		\[
		A_{pq}(\chi_p\chi_q)
		=
		p\sum_x\chi_p(x)\sum_y t_{x,y}\chi_q(y).
		\]
		With $d_x=[xp^{-1}]_q$ and $w=[yp^{-1}]_q$, one has
		\[
		t_{x,y}=[w-d_x]_q=w-d_x+q\,\mathbf1_{\{w<d_x\}},
		\qquad
		\chi_q(y)=\chi_q(p)\chi_q(w).
		\]
		As $y$ runs through the nonzero residues, so does $w$.  The complete
		sum of $\chi_q(w)$ is zero, and therefore
		\[
		\sum_y t_{x,y}\chi_q(y)
		=
		\chi_q(p)
		\left(A_q(\chi_q)+qC_{\chi_q}(d_x)\right).
		\]
		The term involving $A_q(\chi_q)$ vanishes after summation over $x$,
		and \eqref{eq:CRTformula} follows.
	\end{proof}
	
	Combining Lemma~\ref{lem:CRT} and Theorem~\ref{thm:C4basic} gives the
	exact criterion
	\begin{equation}\label{eq:primitivecriterion}
	S_{pq}(a)=\frac{pq\,k}{2}
	\iff
	\Ree\left(
	\chi_q(p)
	\sum_{x=1}^{p-1}
	\chi_p(x)C_{\chi_q}(d_x)
	\right)=0.
	\end{equation}
	
	\subsection{The primitive local type \texorpdfstring{$(4,2)$}{(4,2)}}
	
	We now assume
	$\psi=\chi_p\eps_q$,
	where $\chi_p$ is quartic modulo $p$ and $\eps_q$ is the quadratic
	character modulo $q$.  Put
	$\rho_p(x)=\Ree\chi_p(x)$.
	Thus
	\[
	\rho_p(x)=
	\begin{cases}
	1,&\chi_p(x)=1,\\
	-1,&\chi_p(x)=-1,\\
	0,&\chi_p(x)=\pm i.
	\end{cases}
	\]
	
	\begin{proposition}[Parity and cyclicity]\label{prop:42structure}
		Let $H=\ker(\chi_p\eps_q)$.  Then $\chi_p\eps_q$ is odd if and only if
		\[
		\begin{cases}
		p\equiv1\pmod8,\ q\equiv3\pmod4,
		\end{cases}
		\qquad\text{or}\qquad
		\begin{cases}
		p\equiv5\pmod8,\ q\equiv1\pmod4.
		\end{cases}
		\]
		Moreover, $H$ is cyclic if and only if
		\begin{equation}\label{eq:cyclic42}
		\gcd\left(\frac{p-1}{2},q-1\right)=2.
		\end{equation}
	\end{proposition}
	
	\begin{proof}
		The parity assertion follows from
		$\chi_p(-1)=(-1)^{(p-1)/4}$
		and
		$\eps_q(-1)=(-1)^{(q-1)/2}$.
		For cyclicity, write $p-1=4r$ and $q-1=2s$, and choose generators
		$g_p,g_q$ such that
		\[
		\chi_p(g_p)=i,\qquad \eps_q(g_q)=-1.
		\]
		Then
		\[
		(g_p^u,g_q^v)\in H
		\iff
		u+2v\equiv0\pmod4.
		\]
		Hence the $p$-component of an element of $H$ lies in the subgroup of
		squares, of order $(p-1)/2=2r$, while the $q$-projection is
		surjective.  Therefore the exponent of $H$ is at most
		$\lcm(2r,2s)$.
		Since
		\[
		|H|=\frac{(p-1)(q-1)}4=2rs,
		\]
		cyclicity forces $\gcd(r,s)=1$, equivalently
		\eqref{eq:cyclic42}.  Conversely, if $\gcd(r,s)=1$, then
		$(g_p^2,g_q)\in H$
		has order
		\[
		\lcm(2r,2s)=2rs=|H|,
		\]
		so $H$ is cyclic.
	\end{proof}
	
	Define
	\[
	d_x=[xp^{-1}]_q
	\]
	and
	\[
	E_q(t)=\sum_{1\le n<t}\left(\frac nq\right).
	\]
	
	\begin{theorem}[The $(4,2)$ criterion]\label{thm:42criterion}
		Assume that $H=\ker(\chi_p\eps_q)$ is cyclic and that
		$\chi_p\eps_q$ is odd.  If $a$ is a generator of $H$, then
		\[
		S_{pq}(a)=\frac{pq\,k}{2}
		\]
		if and only if
		\begin{equation}\label{eq:delta42}
		\Delta_{4,2}(p,q)
		:=
		\sum_{x=1}^{p-1}\rho_p(x)E_q(d_x)
		=0.
		\end{equation}
		Moreover,
		\begin{equation}\label{eq:delta42half}
		\Delta_{4,2}(p,q)
		=
		2\sum_{x=1}^{(p-1)/2}\rho_p(x)E_q(d_x).
		\end{equation}
	\end{theorem}
	
	\begin{proof}
		Since $\eps_q(p)=\pm1$ is real, \eqref{eq:primitivecriterion} reduces
		to \eqref{eq:delta42}.  For the second assertion, note that
		$d_{p-x}=[1-d_x]_q$.
		Using
		$E_q([1-d]_q) = -\eps_q(-1)E_q(d)$
		and
		$\rho_p(p-x)=\chi_p(-1)\rho_p(x)$,
		together with
		$\chi_p(-1)\eps_q(-1)=-1$,
		shows that the terms corresponding to $x$ and $p-x$ are equal.
	\end{proof}
	
	\subsubsection{Triangular character sums and an arithmetic consequence}

	For arbitrary nonprincipal characters $\alpha$ modulo $p$ and
	$\beta$ modulo $q$, define
	\begin{equation}\label{eq:Tdef}
	T_{p,q}(\alpha,\beta)
	=\sum_{n=1}^{p-1}\alpha(n)
	\sum_{\substack{1\le y<q\\py<nq}}\beta(y).
	\end{equation}
	The following identities do not require the characters to have
	orders four and two.
    The triangular moment formula obtained by combining
    Lemma~\ref{lem:CRT} with \eqref{eq:Dtriangular} is equivalent to
    \cite[equation~(6)]{SzmidtUrbanowiczZagier}, at moment order one,
    by character orthogonality. The direct proof below records the
    interchange and parity identities, including their endpoints.
	
	\begin{proposition}[Triangular interchange and parity]
		\label{prop:reciprocity}
		Let $p$ and $q$ be distinct odd primes, and let $\alpha$ and $\beta$
		be nonprincipal characters modulo $p$ and $q$, respectively.
		Writing $d_x=[xp^{-1}]_q$, one has
		\begin{equation}\label{eq:Dtriangular}
		\sum_{x=1}^{p-1}\alpha(x)C_\beta(d_x)
		=\alpha(-q)\,T_{p,q}(\alpha,\beta)
		\end{equation}
		and the interchange identity
		\begin{equation}\label{eq:reciprocity}
		T_{p,q}(\alpha,\beta)=-T_{q,p}(\beta,\alpha).
		\end{equation}
		Moreover, with
		\[
		B_{p,q}(\alpha,\beta)=
		\sum_{n=1}^{(p-1)/2}\alpha(n)
		\sum_{1\le y<nq/p}\beta(y),
		\]
		we have
		\begin{equation}\label{eq:triangular-parity}
		T_{p,q}(\alpha,\beta)=
		\begin{cases}
		0,&\alpha(-1)\beta(-1)=1,\\[2pt]
		2B_{p,q}(\alpha,\beta),&\alpha(-1)\beta(-1)=-1.
		\end{cases}
		\end{equation}
	\end{proposition}
	
	\begin{proof}
		For $1\le n\le p-1$, put $x=[-nq]_p$ and
		$c=\lceil nq/p\rceil$. Then $nq+x=pc$, so
		$d_x=[c]_q$. If $c<q$, the inequality $py<nq$ is
		precisely $1\le y<c$. If $c=q$, the two sums are zero:
		$C_\beta(d_x)=C_\beta(0)=0$, while the sum over
		$1\le y<q$ vanishes because $\beta$ is nonprincipal.
		Since $\alpha(x)=\alpha(-q)\alpha(n)$ and the map
		$n\mapsto[-nq]_p$ permutes $1,\ldots,p-1$, this proves
		\eqref{eq:Dtriangular}.
		
		The conditions $py<nq$ and $nq<py$ partition the rectangle
		$1\le n<p$, $1\le y<q$: equality cannot occur since $p$
		and $q$ are distinct primes. The sum of $\alpha(n)\beta(y)$
		over the whole rectangle is zero. This gives
		\eqref{eq:reciprocity}.
		
		Finally, let
		$F(n)=\sum_{1\le y<nq/p}\beta(y)$. Replacing $y$ by
		$q-y$, and using $\sum_{y=1}^{q-1}\beta(y)=0$, gives
		\[
		F(p-n)=-\beta(-1)F(n).
		\]
		As $\alpha(p-n)=\alpha(-1)\alpha(n)$, the summands for
		$n$ and $p-n$ are equal when
		$\alpha(-1)\beta(-1)=-1$, and opposite when the product
		is $1$. Pairing these terms proves
		\eqref{eq:triangular-parity}.
	\end{proof}
	
	In the primitive $(4,2)$ case, specializing
	$\alpha=\chi_p$ and $\beta=\eps_q$ in
	\eqref{eq:Dtriangular} and applying
	Lemma~\ref{lem:CRT} and Theorem~\ref{thm:C4basic} gives
	\begin{equation}\label{eq:deviationT}
	S_{pq}(a)-\frac{pq\,k}{2}
	=\frac{pq}{2}\eps_q(p)
	\Ree\!\left(\chi_p(-q)T_{p,q}(\chi_p,\eps_q)\right).
	\end{equation}
	The parity identity also gives the shorter exact formula
	\begin{equation}\label{eq:deviation-halftriangle}
	S_{pq}(a)-\frac{pq\,k}{2}
	=pq\,\eps_q(p)
	\Ree\!\left(\chi_p(-q)B_{p,q}(\chi_p,\eps_q)\right).
	\end{equation}
	The next consequence applies to all primitive quartic local types,
	not only $(4,2)$.
	
	\begin{corollary}[Triangular evaluation of primitive quartic deviations]
		\label{cor:primitive-quartic-divisibility}
		Let $\psi$ be a primitive character of exact order four modulo
		$pq$, with $p$ and $q$ distinct odd primes. Suppose that
		$H=\ker\psi$ is cyclic, and let $a$ be a generator of $H$.
		Then, with $k=|H|=(p-1)(q-1)/4$,
		\begin{equation}\label{eq:primitive-divisibility}
		S_{pq}(a)\equiv0\pmod{pq}.
		\end{equation}
		If $\psi=\alpha\beta$ is odd, where the nonprincipal local
		characters $\alpha$ and $\beta$ have conductors $p$ and $q$,
		respectively, then more precisely
		\begin{equation}\label{eq:primitive-integer-deviation}
		\frac{S_{pq}(a)}{pq}-\frac{k}{2}
		=\Ree\!\left(
		\beta(p)\alpha(-q)B_{p,q}(\alpha,\beta)
		\right)\in\mathbb Z.
		\end{equation}
	\end{corollary}
	
	\begin{proof}
		Since $\psi$ has order four, at least one of $p-1$ and
		$q-1$ is divisible by four. The other is even; thus
		$k=(p-1)(q-1)/4$ is even. If $\psi(-1)=1$, then
		$-1\in H$ and complementary pairing gives
		$S_{pq}(a)=pq\,k/2$, which proves
		\eqref{eq:primitive-divisibility}.
		
		Suppose now that $\psi(-1)=-1$. Since $\psi$ is primitive,
		its local factors $\alpha$ and $\beta$ are nonprincipal.
		Combining Lemma~\ref{lem:CRT} with
		\eqref{eq:Dtriangular} and \eqref{eq:triangular-parity},
		we obtain
		\[
		A_{pq}(\psi)
		=2pq\,\beta(p)\alpha(-q)B_{p,q}(\alpha,\beta).
		\]
		Theorem~\ref{thm:C4basic} now yields
		\eqref{eq:primitive-integer-deviation}. Because $\psi$ has
		order four, $\alpha$ and $\beta$ take values among
		$0,\pm1,\pm i$. Hence
		$\beta(p)\alpha(-q)B_{p,q}(\alpha,\beta)$
		is a Gaussian integer, and its real part is an integer.
		Since $k$ is even, \eqref{eq:primitive-divisibility} follows.
	\end{proof}
	The divisibility part is an earlier elementary observation, not a
    new consequence of the quartic calculation: Merca~\cite[Section~4, p.~23]{Merca2011}
    states immediately before his Conjecture~1 that whenever
    $(a-1,m)=1$, the residue sum in his equation~(39) is divisible by $m$.
    Indeed, the geometric sum gives
    \[
      (a-1)S_{pq}(a)\equiv a(a^k-1)\equiv0\pmod{pq},
    \]
    hence $pq\mid S_{pq}(a)$. For a primitive quartic kernel, both local
    character orders are even. Their images therefore share $-1$, so
    both projections of the kernel are nontrivial. A generator of a
    cyclic kernel consequently satisfies $(a-1,pq)=1$.
    The half-triangular evaluation in~\eqref{eq:primitive-integer-deviation}
    specializes the classical identity of
    Szmidt--Urbanowicz--Zagier~\cite[equation~(6)]{SzmidtUrbanowiczZagier}
    to the residue-balance problem.
	
\subsubsection{A classical class-number obstruction}
The character-coordinate interpretation of residue deviations is
classical; compare Girstmair~\cite{Girstmair1997,Girstmair2026}.
Here the relative class-number formula supplies a necessary condition
for any coset of a primitive quartic kernel to have centered sum.

\begin{proposition}\label{prop:class-number-obstruction}
Let $\psi$ be a primitive odd quartic character modulo $N=pq$,
let $H=\ker\psi$, and choose $g$ with $\psi(g)=i$.
For a coset $C$ of $H$ put
\[
 D(C)=\sum_{a\in C}\Res{a}{N}-\frac{N|H|}{2},
 \qquad x=\frac{D(H)}N,\quad y=\frac{D(gH)}N.
\]
Let $F=\mathbb Q(\zeta_N)^H$, $F^+=F\cap\mathbb R$, and
$h^-(F)=h(F)/h(F^+)$, where $h$ denotes the ideal class number.
Then $x,y\in\Z$, the four coset deviations are
\begin{equation}\label{eq:quartic-norm-pattern}
 \bigl(D(H),D(gH),D(g^2H),D(g^3H)\bigr)=N(x,y,-x,-y),
\end{equation}
and
\begin{equation}\label{eq:quartic-class-number}
 h^-(F)=2(x^2+y^2).
\end{equation}
If $H$ is sum-balanced, then $h^-(F)=2b^2$ for some integer $b\ge1$,
and its two noncentered cosets have deviations of absolute value $Nb$.
If $h^-(F)/2$ is not an integer square, no coset is centered.
No cyclicity assumption on $H$ is required.
\end{proposition}
\begin{proof}
Write $\psi=\alpha\beta$ in its nonprincipal local characters.
Lemma~\ref{lem:CRT} and Proposition~\ref{prop:reciprocity} give
\[
 \frac{A_N(\psi)}{2N}
 =\beta(p)\alpha(-q)B_{p,q}(\alpha,\beta)\in\Z[i].
\]
Orthogonality on the coset $g^jH$ gives
$D(g^jH)=\frac12\Ree((-i)^jA_N(\psi))$.
Consequently $A_N(\psi)/(2N)=x+iy$, which proves integrality
and \eqref{eq:quartic-norm-pattern}.

The field $F$ is imaginary cyclic quartic of conductor $N$, and
its odd characters are $\psi,\overline\psi$. The classical relative
class-number formula and its cyclic-field constants give
\[
 h^-(F)=\frac{Q_Fw_F}{4}L(0,\psi)L(0,\overline\psi),
 \qquad Q_F=1,\quad w_F=2;
\]
see Louboutin~\cite{Louboutin2026}, with the formula and constants in
\href{https://arxiv.org/abs/2508.10563v1}{arXiv:2508.10563v1},
p.~2, equations~(1)--(3).
The only cyclic quartic exception to $w_F=2$ is
$\mathbb Q(\zeta_5)$, excluded by the primitive conductor $pq$.
Since $L(0,\psi)=-A_N(\psi)/N$, this yields
$h^-(F)=|A_N(\psi)/N|^2/2=2(x^2+y^2)$.
If $x=0$, positivity of the relative class number forces $y\ne0$;
put $b=|y|$. More generally, a centered coset forces $x=0$ or $y=0$,
so $h^-(F)/2$ must be an integer square.
\end{proof}

The square condition is necessary; it does not determine which
coordinate, if any, vanishes. For the characters
$\psi=\chi_5\eps_q$, normalized by $\chi_5(2)=i$, the exact moments give
\[
 \begin{array}{c|c|c}
 N&(x,y)&h^-(F)\\ \hline
 85&(0,1)&2\\
 185&(-2,1)&10
 \end{array}
\]
The latter value excludes a centered coset because $10/2=5$ is not
a square. These class-number values follow from the classical formula
applied to the displayed exact moments.

	\subsection{Small quartic primes}
	
	\subsubsection{The case \texorpdfstring{$p=5$}{p=5}}
	
	Let $\chi_5$ be the quartic character modulo $5$ normalized by
	$\chi_5(2)=i$.  Then
	$\chi_5(-1)=-1$.
	Hence the primitive type $(4,2)$ is odd precisely when
	$q\equiv1\pmod4$.
	Moreover,
	\[
	\gcd\left(\frac{5-1}{2},q-1\right)=2,
	\]
	so the kernel is automatically cyclic.
	
	\begin{corollary}\label{cor:p5}
		Let $q\ne5$ be a prime with $q\equiv1\pmod4$, let
		$H=\ker(\chi_5\eps_q)$,
		and let $a$ be any generator of $H$.  Then
		\[
		k=q-1
		\]
		and
		\[
		S_{5q}(a)=\frac{5q(q-1)}2
		\]
		if and only if
		\[
		\sum_{n=1}^{\lfloor q/5\rfloor}
		\left(\frac nq\right)=0
		\]
		when
		\[
		q\equiv\pm1\pmod5,
		\]
		and if and only if
		\[
		\sum_{n=1}^{\lfloor 2q/5\rfloor}
		\left(\frac nq\right)=0
		\]
		when
		\[
		q\equiv\pm2\pmod5.
		\]
	\end{corollary}
	
	\begin{proof}
		Since
		$(\rho_5(1),\rho_5(2))=(1,0)$,
		Theorem~\ref{thm:42criterion} gives
		\[
		S_{5q}(a)=\frac{5q(q-1)}2
		\iff
		E_q(5^{-1}\bmod q)=0.
		\]
		The inverse of $5$ modulo $q$ is
		\[
		5^{-1}\equiv
		\begin{cases}
		(4q+1)/5,&q\equiv1\pmod5,\\
		(2q+1)/5,&q\equiv2\pmod5,\\
		(3q+1)/5,&q\equiv3\pmod5,\\
		(q+1)/5,&q\equiv4\pmod5.
		\end{cases}
		\]
		Since $q\equiv1\pmod4$, the quadratic character is even, and the
		symmetry
		\[
		E_q(q+1-r)=-E_q(r)
		\]
		gives the stated formulas.
	\end{proof}

The same quartic--quadratic product character and the Legendre sums
ending at $q/5$ and $2q/5$ occur in Girstmair's treatment of digit
frequencies and class numbers, especially Section~4 of the preprint
\cite{GirstmairDigits2026}. Those character coordinates are therefore
classical; here their vanishing is interpreted as the balance test for
the specified cyclic kernel.
	
\begin{corollary}\label{cor:p5-parity}
Under the hypotheses of Corollary~\ref{cor:p5}, sum-balance requires
$q\equiv1$ or $17\pmod{20}$.
\end{corollary}
\begin{proof}
For $q\equiv9\pmod{20}$, the sum with $\lfloor q/5\rfloor$ terms
has odd length. For $q\equiv13\pmod{20}$, the sum with
$\lfloor2q/5\rfloor$ terms has odd length. Each summand is $1$ or
$-1$, so neither sum can vanish. These exclude two of the four
classes allowed by $q\equiv1\pmod4$ and $q\ne5$.
\end{proof}

	\begin{example}
		For $q=17$ one may take $a=14$ modulo $85$.  Then
		\[
		k=16,\qquad
		S_{85}(14)=680=\frac{85\cdot16}{2}.
		\]
		Indeed,
		\[
		\sum_{n=1}^{6}\left(\frac n{17}\right)=0.
		\]
	\end{example}
	
	\begin{example}
		For $q=37$ one may take $a=39$ modulo $185$.  Then
		\[
		k=36,\qquad
		S_{185}(39)=2960\ne3330=\frac{185\cdot36}{2}.
		\]
		Here
		\[
		\sum_{n=1}^{14}\left(\frac n{37}\right)=2.
		\]
		Thus even the congruence class of $q$ modulo $20$ does not determine
		the outcome: $17\equiv37\pmod{20}$, but the two examples behave
		differently.
	\end{example}
	
	\subsubsection{The case \texorpdfstring{$p=13$}{p=13}}
	
	Normalize the quartic character modulo $13$ by $\chi_{13}(2)=i$.
	Then
	\[
	(\rho_{13}(1),\dots,\rho_{13}(6))
	=
	(1,0,1,-1,0,0).
	\]
	Since $13\equiv5\pmod8$, the oddness condition requires
	$q\equiv1\pmod4$, while cyclicity requires
	$\gcd(6,q-1)=2$.
	For primes $q>3$ these conditions are equivalent to
	$q\equiv5\pmod{12}$.
	
	\begin{corollary}\label{cor:p13}
		Let $q\ne13$ be a prime with $q\equiv5\pmod{12}$ and put
		$d_r=[13^{-1}r]_q$.
		If $a$ is a generator of
		$\ker(\chi_{13}\eps_q)$,
		then
		\[
		S_{13q}(a)=\frac{13q\,k}{2}
		\qquad
		\text{if and only if}
		\qquad
		E_q(d_1)+E_q(d_3)-E_q(d_4)=0.
		\]
	\end{corollary}
	
	\begin{example}
		For $q=29$ one may take $a=10$ modulo $377$.  Then
		\[
		k=84,\qquad
		S_{377}(10)=15834=\frac{377\cdot84}{2}.
		\]
		For $q=53$ one may take $a=56$ modulo $689$.  Then
		\[
		k=156,\qquad
		S_{689}(56)=55120\ne53742=\frac{689\cdot156}{2}.
		\]
	\end{example}
	
	\subsection{The primitive local type \texorpdfstring{$(4,4)$}{(4,4)}}\label{sec:primitive44}

	For completeness, suppose now that
	$\psi=\chi_p\chi_q$
	with both local characters quartic, and put
	$H=\ker\psi$.
	
	\begin{proposition}[Cyclicity in the \((4,4)\) case]
		\label{prop:44cyclic}
		The subgroup \(H\) is cyclic if and only if
		\[
		\gcd\left(\frac{p-1}{4},\frac{q-1}{4}\right)=1.
		\]
	\end{proposition}
	
	\begin{proof}
		Write
		$p-1=4r$, $q-1=4s$,
		and choose generators \(g_p\) and \(g_q\) such that
		$\chi_p(g_p)=\chi_q(g_q)=i$.
		Then
		\[
		(g_p^u,g_q^v)\in H
		\iff
		u+v\equiv0\pmod4.
		\]
		Moreover,
		\[
		|H|=\frac{(p-1)(q-1)}4=4rs.
		\]
		
		Every element of \(H\) has order dividing
		$\lcm(4r,4s)=4\,\lcm(r,s)$.
		Hence, if \(H\) is cyclic, its exponent must equal its order, and therefore
		$4rs\le 4\,\lcm(r,s)$,
		which forces
		$\gcd(r,s)=1$.
		
		Conversely, if \(\gcd(r,s)=1\), then
		$(g_p,g_q^{-1})\in H$ and this element has order
		$\lcm(4r,4s)=4rs=|H|$.
		Thus \(H\) is cyclic.
	\end{proof}
	
	The oddness condition is
	\[
	\chi_p(-1)\chi_q(-1)=-1,
	\]
	equivalently
	\[
	\{p\bmod8,q\bmod8\}=\{1,5\}.
	\]
	
	Define
	\[
	D_{4,4}(\chi_p,\chi_q)
	=
	\sum_{x=1}^{p-1}
	\chi_p(x)C_{\chi_q}([xp^{-1}]_q).
	\]
	Assume that $H=\ker(\chi_p\chi_q)$ is cyclic, that
	$\chi_p\chi_q$ is odd, and that $a$ is a generator of $H$.
	Then $H^\perp=\langle\chi_p\chi_q\rangle\simeq C_4$, and
	Lemma~\ref{lem:CRT} together with Theorem~\ref{thm:C4basic} gives
	\[
	S_{pq}(a)=\frac{pq\,k}{2}
	\iff
	\Ree\bigl(\chi_q(p)D_{4,4}(\chi_p,\chi_q)\bigr)=0.
	\]
	If
	$D_{4,4}(\chi_p,\chi_q)=U+iV$,
	this becomes
	\[
	\begin{array}{c|c}
	\chi_q(p)&\text{condition}\\ \hline
	\pm1&U=0,\\
	\pm i&V=0.
	\end{array}
	\]
	Thus, unlike the type $(4,2)$, the primitive type $(4,4)$ naturally
	involves genuinely complex partial character sums.

	\section{Subgroups of index six}

	We begin with the group-theoretic constraints on cyclic subgroups
	of index six.  Besides determining the admissible moduli, these
	constraints give an exact count of the mixed cyclic subgroups,
    independently of their residue sums. The total counts specialize
    the rank-two formulas of Tóth~\cite[Theorems~4.6--4.7]{Toth2014};
    the split according to the diagonal involution identifies the
    cases needed for the arithmetic investigation.
	
	\begin{theorem}[Existence and enumeration at index six]
		\label{thm:index6-structure}
		Let $p$ and $q$ be distinct odd primes, put $m=p-1$, $n=q-1$,
		and write $G=(\mathbb Z/pq\mathbb Z)^\times$.  A cyclic subgroup
		of index six exists in $G$ if and only if
		\begin{equation}\label{eq:index6-moduli}
		3\mid mn\qquad\text{and}\qquad \gcd(m,n)\mid 6.
		\end{equation}
		Assume these conditions and set
		$u=\nu_3(m)$, $v=\nu_3(n)$ and
		\[
		M_3(p,q)=
		\begin{cases}
		1,&\min(u,v)=0,\\
		4,&u=v=1,\\
		3,&\min(u,v)=1<\max(u,v).
		\end{cases}
		\]
		Then the number of cyclic subgroups of index six is
		\[
		N_6^{\mathrm{cyc}}(p,q)=
		\begin{cases}
		3M_3(p,q),&\nu_2(m)=\nu_2(n)=1,\\
		2M_3(p,q),&\text{otherwise},
		\end{cases}
		\]
		while the number of such subgroups not containing $-1$ is
		\begin{equation}\label{eq:index6-mixed-count}
		\boxed{N_6^{\mathrm{mix}}(p,q)=2M_3(p,q).}
		\end{equation}
		In particular, whenever cyclic index-six subgroups exist, the
		number of mixed ones is exactly $2$, $6$, or $8$.
	\end{theorem}
	
	\begin{proof}
		Under the Chinese remainder theorem,
		$G\simeq C_m\times C_n$.  If $H$ is cyclic of index six, then
		$|H|=mn/6$ divides the exponent $\operatorname{lcm}(m,n)$ of $G$.
		Hence $3\mid mn$ and $\gcd(m,n)\mid6$ are necessary.
		
		Conversely, assume \eqref{eq:index6-moduli}.  A subgroup of index
		six has Sylow subgroup of index $2$ in the $2$-Sylow subgroup of $G$,
		of index $3$ in the $3$-Sylow subgroup, and is the full Sylow
		subgroup for all other primes.  Because $\gcd(m,n)\mid6$, each
		of those other Sylow subgroups is cyclic.
		
		Write $a=\nu_2(m)$ and $b=\nu_2(n)$.  Both are positive, and
		$\min(a,b)=1$.  If $a=b=1$, the $2$-Sylow subgroup is
		$C_2\times C_2$ and has three subgroups of index two, all cyclic.
		Exactly two omit its diagonal element $(-1,-1)$.  If, say,
		$a>1=b$, the $2$-Sylow subgroup is $C_{2^a}\times C_2$.
		A kernel of a surjection to $C_2$ is cyclic exactly when the
		surjection is nontrivial on the second factor.  There are two
		such kernels, and both omit $(-1,-1)$; the case $b>1=a$ is
		symmetric.  Thus there are always exactly two admissible
		mixed choices for the $2$-primary component.
		
		For the $3$-primary component, the group is
		$C_{3^u}\times C_{3^v}$, and its required subgroup has index
		three.  If $\min(u,v)=0$, this group is cyclic and has exactly
		one subgroup of index three.  If $u=v=1$, it is
		$C_3\times C_3$ and has four subgroups of index three,
		all cyclic.  If, say, $u>1=v$, a kernel of a surjection onto
		$C_3$ is cyclic exactly when the surjection is nontrivial on
		the $C_3$ factor; there are three such kernels.  The opposite
		case is symmetric.  These are precisely the values of $M_3(p,q)$.
		
		The choices at distinct primes combine uniquely into a subgroup,
		and their direct product is cyclic because the Sylow orders are
		coprime.  The element $-1$ lies entirely in the $2$-primary
		component.  Multiplying the counts proves both formulas and
		also constructs a mixed cyclic subgroup whenever
		\eqref{eq:index6-moduli} holds.
	\end{proof}
	
	The same Sylow analysis gives an exact criterion in terms of
	the local components of a sextic character.
	
	\begin{proposition}[Local-character cyclicity criterion]
		\label{prop:index6-character-cyclicity}
		Let $\psi=\chi_p\chi_q$ be an odd Dirichlet character of exact
		order six modulo $pq$, with local characters $\chi_p$ and $\chi_q$.
		Put $u=\nu_3(p-1)$ and $v=\nu_3(q-1)$.
		Then $H=\ker\psi$ is cyclic if and only if
		\begin{enumerate}[label=\textup{(\roman*)}]
			\item $\gcd(p-1,q-1)\mid6$; and
			\item if $u,v>0$ and $u\ne v$, the local character at the prime
			for which the $3$-adic valuation is $1$ has order divisible by $3$.
		\end{enumerate}
	\end{proposition}
	
	\begin{proof}
		Since $\psi$ has order six, its kernel has index six.
		The necessity of (i) follows from Theorem~\ref{thm:index6-structure}.
		In the $3$-Sylow subgroup, cyclicity is automatic if $\min(u,v)=0$
		or $u=v=1$.  If $u>1=v$, the preceding proof shows that the
		kernel is cyclic precisely when the order-three part of $\psi$
		is nontrivial on the $C_3$ factor belonging to $q$; this is
		equivalent to $3\mid\operatorname{ord}(\chi_q)$.  The case
		$v>1=u$ is symmetric.
		
		At the prime $2$, condition (i) gives
		$\min(\nu_2(p-1),\nu_2(q-1))=1$.  If both valuations are one,
		oddness of $\psi$ selects one of the two order-two kernels
		omitting $(-1,-1)$.  If only one valuation is one, oddness forces
		the order-two part of $\psi$ to be nontrivial on that smaller
		local factor, again making the kernel cyclic.  All other Sylow
		subgroups are cyclic under (i).  The two conditions are therefore
		sufficient as well.
	\end{proof}
	
	\begin{corollary}[Separated quadratic--cubic local type]
		\label{cor:index6-separated}
		Suppose $p\equiv3\pmod4$ and $q\equiv1\pmod3$, and let
		$\theta_q$ be any character of exact order three modulo $q$.
		For the odd sextic character $\psi=\varepsilon_p\theta_q$,
		the kernel $H$ is cyclic if and only if
		\begin{equation}\label{eq:index6-separated}
		\gcd\left(\frac{p-1}{2},\frac{q-1}{3}\right)=1.
		\end{equation}
	\end{corollary}
	
	\begin{proof}
		The quadratic and cubic characters have coprime orders, so under
		CRT the kernel is
		$\ker\varepsilon_p\times\ker\theta_q
		\simeq C_{(p-1)/2}\times C_{(q-1)/3}$.
		Such a direct product is cyclic exactly when its two factors
		have coprime orders.
	\end{proof}
	\begin{remark}\label{rem:index6-census-compatibility}
		Theorem~\ref{thm:index6-structure} counts all cyclic index-six
		subgroups, including those with a trivial projection to one local
		factor.  By contrast, the census in Section~\ref{sec:experiments}
		retains only subgroups whose two local projection orders exceed one.
		For example, modulo $91=7\cdot13$ the theorem gives eight mixed
		cyclic subgroups of index six.  Direct enumeration shows that five
		are sum-balanced and three are not.  Thus the structural enumeration
		does not, by itself, decide the residue-sum equality.
	\end{remark}

        For reference, direct enumeration gives the following small table;
        the first two columns count all cyclic index-six subgroups and the
        last column counts the mixed ones that are sum-balanced.
        \begin{center}
        \begin{tabular}{c|r|r|r}
        $pq$ & all cyclic & mixed & mixed, sum-balanced\\ \hline
        $35$ & 2 & 2 & 1\\
        $91$ & 8 & 8 & 5\\
        $133$ & 9 & 6 & 2\\
        $217$ & 12 & 8 & 1
        \end{tabular}
        \end{center}
	
	We now turn to the residue-sum criterion for
	$[G:H]=6$,
	and assume that $-1\notin H$.  
	Since $H^\perp$ is an abelian group of squarefree order $6$, it is
	necessarily cyclic.  Write
	$H^\perp=\langle\psi\rangle\simeq C_6$.
	The image of $-1$ in $G/H$ is the unique element of order $2$.
	Since $\psi$ generates the character group of $G/H$, it takes the
	value $-1$ on this element.  Consequently,
	$\psi(-1)=-1$.
	Thus the odd characters in $H^\perp$ are precisely
	$\psi$, $\psi^3$, $\psi^5=\overline{\psi}$.

	\begin{theorem}[The index-six criterion]\label{thm:index6}
		Suppose that $[G:H]=6$ and $-1\notin H$, and let
		$H^\perp=\langle\psi\rangle$.
		Then
		\begin{equation}\label{eq:index6deviation}
		S_{pq}(a)-\frac{pq\,k}{2}
		=
		\frac13\Ree A_{pq}(\psi)
		+\frac16 A_{pq}(\psi^3).
		\end{equation}
		In particular,
		\begin{equation}\label{eq:index6criterion}
		S_{pq}(a)=\frac{pq\,k}{2}
		\iff
		2\Ree A_{pq}(\psi)+A_{pq}(\psi^3)=0.
		\end{equation}
		Equivalently,
		\[
		S_{pq}(a)=\frac{pq\,k}{2}
		\iff
		2\Ree L(0,\psi)+L(0,\psi^3)=0.
		\]
	\end{theorem}
	
	\begin{proof}
		By Theorem~\ref{thm:character-criterion},
		\[
		S_{pq}(a)-\frac{pq\,k}{2}
		=
		\frac16
		\left(
		A_{pq}(\psi)
		+
		A_{pq}(\psi^3)
		+
		A_{pq}(\psi^5)
		\right).
		\]
		Since
		$\psi^5=\overline{\psi}$
		and
		$A_{pq}(\overline{\psi}) = \overline{A_{pq}(\psi)}$,
		we obtain
		\[
		A_{pq}(\psi)+A_{pq}(\psi^5)
		=
		2\Ree A_{pq}(\psi),
		\]
		which gives \eqref{eq:index6deviation}.  The criterion
		\eqref{eq:index6criterion} follows immediately.  Finally, using
		$A_{pq}(\chi)=-pq\,L(0,\chi)$
		for each nonprincipal character involved gives the equivalent
		$L$-value formulation.
	\end{proof}
	
	The character
	$\eta:=\psi^3$
	is the unique nontrivial quadratic character in $H^\perp$.  Since
	$\psi(-1)=-1$, it is also odd.  Hence
	$\eta\in\{\eps_p,\eps_q,\eps_p\eps_q\}$.
	The quadratic contribution in \eqref{eq:index6criterion} can therefore
	be made completely explicit.
	
	\begin{corollary}\label{cor:index6quadratic}
		Under the hypotheses of Theorem~\ref{thm:index6}, the following hold.
		\begin{enumerate}[label=\textup{(\roman*)}]
			\item If $\psi^3=\eps_p$, then
			\[
			S_{pq}(a)=\frac{pq\,k}{2}
			\qquad
			\text{if and only if}
			\qquad
			2\Ree A_{pq}(\psi)
			+
			q\left(1-\left(\frac qp\right)\right)A_p(\eps_p)
			=0.
			\]
			
			\item If $\psi^3=\eps_q$, then
			\[
			S_{pq}(a)=\frac{pq\,k}{2}
			\qquad
			\text{if and only if}
			\qquad
			2\Ree A_{pq}(\psi)
			+
			p\left(1-\left(\frac pq\right)\right)A_q(\eps_q)
			=0.
			\]
			
			\item If $\psi^3=\eps_p\eps_q$, then
			\[
			S_{pq}(a)=\frac{pq\,k}{2}
			\qquad
			\text{if and only if}
			\qquad
			2\Ree A_{pq}(\psi)
			-pq\,L(0,\eps_p\eps_q)
			=0.
			\]
		\end{enumerate}
	\end{corollary}
	
	\begin{proof}
		The formulas follow from \eqref{eq:index6criterion} and
		Lemma~\ref{lem:induced-moment}, together with
		\[
		A_{pq}(\eps_p)
		=
		q\left(1-\left(\frac qp\right)\right)A_p(\eps_p),
		\]
		\[
		A_{pq}(\eps_q)
		=
		p\left(1-\left(\frac pq\right)\right)A_q(\eps_q),
		\]
		and
		\[
		A_{pq}(\eps_p\eps_q)
		=
		-pq\,L(0,\eps_p\eps_q).
		\]
	\end{proof}
	
	\begin{remark}
		The index-six case exhibits a new feature.  At index $2$ the deviation
		is governed by a single quadratic character, whereas at index $6$ the
		quadratic contribution may cancel the real part of the contribution
		from a character of order $6$.  Thus the nonvanishing of
		$A_{pq}(\psi^3)$ alone does not exclude the sum-balanced identity.
	\end{remark}
	
	\begin{example}\label{ex:index6-cancellation}
		Let $\zeta=e^{\pi i/3}$, and let $\theta$ be the character modulo $7$
		defined by $\theta(3)=\zeta$. The primitive character
		$\psi=\eps_5\theta$ modulo $35$ is odd and has order $6$.
		Since $13\equiv3\pmod5$ and $13\equiv6=3^3\pmod7$, we have
		$\psi(13)=1$. Its kernel has order $4$, and hence
		\[
		\ker\psi=\langle13\rangle=\{1,13,29,27\}.
		\]
		The first moments can be evaluated by grouping the unit
		representatives according to their character values:
		\[
		\begin{array}{c|rrrrrr}
		\psi(n)&1&\zeta&\zeta^2&-1&\zeta^4&\zeta^5\\ \hline
		\displaystyle\sum_{\substack{1\le n<35\\\psi(n)=\zeta^j}}n
		&70&105&70&70&35&70
		\end{array}
		\]
		Here the columns are indexed by $j=0,\ldots,5$. Thus
		\[
		A_{35}(\psi)=70\zeta=35+35\sqrt3\,i,
		\qquad A_{35}(\psi^3)=-70,
		\]
		so the two nonzero terms cancel:
		\[
		2\Ree A_{35}(\psi)+A_{35}(\psi^3)=0.
		\]
		Accordingly,
		$S_{35}(13)=13+29+27+1=70=35\cdot4/2$.
		The local base orders are $(4,2)$, the half-order power is $29$,
		and the subgroup index is $6$; the local character orders are $(2,6)$.
	\end{example}
	
	\begin{example}
		The modulus
		\[
		91=7\cdot13
		\]
		already exhibits both possibilities at index $6$.  For $a=2$ and
		$a=11$ one has
		\[
		\ord_7(a)=3,\qquad \ord_{13}(a)=12,
		\qquad k=12,
		\]
		and in both cases
		\[
		a^{k/2}\equiv(1,-1)\pmod{(7,13)}.
		\]
		Moreover,
		\[
		[(\Z/91\Z)^\times:\langle a\rangle]
		=
		\frac{72}{12}=6.
		\]
		Nevertheless,
		\[
		S_{91}(2)=364\ne546=\frac{91\cdot12}{2},
		\]
		whereas
		\[
		S_{91}(11)=546=\frac{91\cdot12}{2}.
		\]
		Thus, even at fixed index $6$, the local orders and the mixed
		half-order sign pattern do not determine the sum-balanced property.
	\end{example}

\begin{theorem}[An infinite sextic cancellation family]\label{thm:sextic19family}
Let $q>3$ be prime with $q\equiv2\pmod3$, and put
\[
 K=\{1,7,11\}\le\F_{19}^\times,\qquad
 H=K\times\F_q^\times\le(\Z/19q\Z)^\times.
\]
Then $H$ is cyclic, omits $-1$, and has index six and order $3(q-1)$.
Every generator $a$ satisfies $(a-1,19q)=1$.
Let $\zeta=e^{\pi i/3}$, normalize the sextic character $\chi$ modulo
$19$ by $\chi(2)=\zeta$, and write $\chi(q)=\zeta^j$.
Then
\begin{equation}\label{eq:sextic19deviation}
 S_{19q}(a)-\frac{19q\,3(q-1)}2
 =\begin{cases}0,&j\in\{0,4,5\},\\-19q,&j\in\{1,2,3\}.
 \end{cases}
\end{equation}
Every coset is centered exactly when $j=0$.
If $j=5$, the character $\widetilde\chi$ induced modulo $19q$ satisfies
\begin{equation}\label{eq:sextic19moments}
 A_{19q}(\widetilde\chi)=38q(1-\zeta),\qquad
 A_{19q}(\widetilde\chi^{\,3})=-38q.
\end{equation}
Both moments are nonzero and cancel in the index-six criterion.
In particular, every prime $q\equiv29\pmod{57}$ gives such a
cancellation, so there are infinitely many examples.
\end{theorem}
\begin{proof}
The product $C_3\times C_{q-1}$ is cyclic because $(3,q-1)=1$.
It omits $-1$, and its generators have both local orders greater
than one. The six local cosets and their sums are
\[
\begin{array}{c|c|c}
 j&2^jK&s_j\\ \hline
 0&\{1,7,11\}&19\\
 1&\{2,3,14\}&19\\
 2&\{4,6,9\}&19\\
 3&\{8,12,18\}&38\\
 4&\{5,16,17\}&38\\
 5&\{10,13,15\}&38
\end{array}
\]
Lemma~\ref{lem:coset-lifting} gives the deviation of $H$ as
$q(s_0-s_{-j})$, with subscripts taken modulo six. This proves
\eqref{eq:sextic19deviation}. The vector $(s_0,\ldots,s_5)$ has
no nonzero rotational period; the same lemma shows that all lifted
cosets are centered precisely when $j=0$.

The six sums also give
$A_{19}(\chi)=-38\zeta$ and $A_{19}(\chi^3)=-19$.
By Lemma~\ref{lem:induced-moment},
\[
 A_{19q}(\widetilde\chi)=-38q(1-\zeta^j)\zeta,\qquad
 A_{19q}(\widetilde\chi^{\,3})=-19q(1-(-1)^j).
\]
For $j=5$, the identities $\zeta^5=1-\zeta$ and
$\zeta^2=\zeta-1$ yield \eqref{eq:sextic19moments}, and hence
\[
 2\Ree A_{19q}(\widetilde\chi)+A_{19q}(\widetilde\chi^{\,3})
 =38q-38q=0.
\]
Finally, $q\equiv29\pmod{57}$ implies $q\equiv2\pmod3$ and
$q\equiv10\pmod{19}$, with $10\in2^5K$.
Since $(29,57)=1$, Dirichlet's theorem supplies infinitely many
primes in this progression.
\end{proof}

When $j=5$, the six cosets $2^\ell H$, $\ell=0,\ldots,5$, have
deviations $19q(0,0,-1,0,0,1)$, so four are centered.
The first prime in the progression is $q=29$. The element $a=11$
generates $H$ modulo $551$, with
\[
 \ord_{551}(11)=84,\qquad S_{551}(11)=23142=\frac{551\cdot84}{2}.
\]
This infinite family has an imprimitive sextic annihilator; it does
not establish infinitude for primitive sextic annihilators such as
the one in Example~\ref{ex:index6-cancellation}.

	\section{Computational evidence}\label{sec:experiments}

	The computational supplement contains
    \nolinkurl{residue_sum_census.py}, \nolinkurl{data/subgroups.csv.gz},
    and \nolinkurl{data/checks.json}. The dependency-free Python~3 program,
    run with \texttt{--limit 10000 --output data}, regenerates one record
    per eligible cyclic subgroup, including its least positive generator,
    local orders, index, subgroup sum, and, at index four, all four coset
    sums and their balance test. It checks both counting identities below
    and the predicted multiplicity $\varphi(\gcd(u,v))$ for each pair of
    local orders. The companion \nolinkurl{verify_v8.py} checks the
    structural and arithmetic statements over the larger ranges recorded
    in \nolinkurl{verification-results.json}.
The additional \nolinkurl{verify_v12.py} checks the general lifting
identity on $6\,948$ cosets at distinct odd primes at most $43$;
all $1\,023$ lifts with quartic conductor $5<r\le500$, $r\equiv5\pmod8$,
and odd second prime $q\le200$, $q\ne r$; all $2\,346$ primitive odd
quartic characters with $pq\le10000$; and the $337$ primes
$3<q\le5000$ in the sextic family's range. The primitive checks include
$398$ character cases with noncyclic kernel, representing $199$ distinct
noncyclic kernels. These are exact integer computations, recorded
in \nolinkurl{v12-results.json}; the class-number values there are
obtained from the classical formula, not from a separate class-group
computation. The package includes \nolinkurl{archive-reference.json},
reproduction commands, software versions, and checksums.
	The separate Lean formalizations cited in the introduction have the
	more limited scope stated there.

	\subsection{Subgroup census and validation}
	
	We enumerate all distinct odd-prime products $p<q$, $pq\le10000$,
	and all cyclic subgroups $H$ of even order whose generators satisfy
	$(a-1,pq)=1$. There are $1932$ prime products and $75\,090$ eligible
	subgroups. Eligibility is a subgroup property: $(a-1,pq)=1$ is equivalent
	to both local orders being greater than one. Each subgroup is represented
	by its least positive generator, and its residue sum includes the final
	power $a^k\equiv1$.
	
	Table~\ref{tab:index-counts} counts each subgroup once. All $21\,031$
	subgroups containing $-1$ are sum-balanced. Among the $54\,059$ mixed
	subgroups, $11\,988$ are sum-balanced and $42\,071$ are unbalanced.
	Thus the centered identity occurs in about $22.2\%$ of the mixed
	subgroups in this range. Counting their eligible generators instead gives
	$817\,158$ balanced and $4\,054\,994$ unbalanced mixed bases, or about
	$16.8\%$ balanced. The difference reflects the weight $\varphi(k)$
	attached to a subgroup of order $k$; the base totals are
	$1\,311\,354$ direct and $6\,183\,506$ overall. Both counting conventions
	are provided in the supplement, with full tables at every index.
	
	\begin{table}[htbp]
		\centering
		\begin{tabular}{c|r|r|r|r}
			$d$ & Direct & Mixed, balanced & Mixed, unbalanced & Total\\ \hline
			$2$ & $470$ & $610$ & $1414$ & $2494$\\
			$4$ & $721$ & $244$ & $1840$ & $2805$\\
			$6$ & $626$ & $260$ & $2447$ & $3333$\\
			$8$ & $412$ & $145$ & $1180$ & $1737$\\ \hline
			All & $21\,031$ & $11\,988$ & $42\,071$ & $75\,090$
		\end{tabular}
		\caption{Cyclic subgroup counts for distinct odd primes $p<q$,
			$pq\le10000$, even order, and local orders greater than one.
			Here $d=[G:H]$ and ``direct'' means $-1\in H$.}
		\label{tab:index-counts}
	\end{table}
	
	Two independent coverage formulas check the enumeration. Write
	$\operatorname{odd}(n)=n/2^{\nu_2(n)}$. The number of eligible even-order
	bases is
	\[
	\begin{split}
	&\sum_{\substack{p<q,\ pq\le10000\\p,q\text{ odd primes}}}
	\Bigl[(p-2)(q-2)
	-(\operatorname{odd}(p-1)-1)(\operatorname{odd}(q-1)-1)\Bigr]\\
	&=6\,183\,506.
	\end{split}
	\]
	Indeed, CRT leaves $(p-2)(q-2)$ allowed coordinates, and the subtracted
	term counts those with both local orders odd. Separately,
	Lemma~\ref{lem:linkage} gives the number of eligible cyclic subgroups as
	\[
	\sum_{\substack{p<q,\ pq\le10000\\p,q\text{ odd primes}}}
	\ \sum_{\substack{u\mid p-1,\ v\mid q-1\\u,v>1,\ 2\mid\lcm(u,v)}}
	\varphi(\gcd(u,v))=75\,090.
	\]
	The program checks both formulas and verifies the linkage count within
	every represented pair of local orders.
	The analogous divisor-sum count when
		$(a-1,pq)=1$ is removed is $93\,815$ even cyclic subgroups; this
		broader total is also calculated by the program, but the accompanying
		CSV is restricted to the $75\,090$ eligible subgroups.
	
	At index four, the aggregate count is clarified by the local character
	types in Table~\ref{tab:quartic-counts}. For $H^\perp\simeq C_4$, a
	generator $\psi=\chi_p\chi_q$ and its conjugate have the same local
	orders, so this type is attached to the subgroup. The table combines
	types obtained by interchanging the primes; the data retain their
	orientation $p<q$.
	
	\begin{table}[htbp]
		\centering
		\begin{tabular}{c|c|r|r|r|r}
			$H^\perp$ & Local type & Direct & Mixed, balanced & Mixed, unbalanced & Total\\ \hline
			$C_2\times C_2$ & --- & $0$ & $0$ & $545$ & $545$\\
			$C_4$ & $(4,1),(1,4)$ & $0$ & $207$ & $358$ & $565$\\
			$C_4$ & $(4,2),(2,4)$ & $479$ & $28$ & $656$ & $1163$\\
			$C_4$ & $(4,4)$ & $242$ & $9$ & $281$ & $532$\\ \hline
			& All & $721$ & $244$ & $1840$ & $2805$
		\end{tabular}
		\caption{The eligible index-four subgroups, stratified by annihilator
			and the local orders of its quartic generator.}
		\label{tab:quartic-counts}
	\end{table}
	
	All $545$ biquadratic cases are unbalanced, as
	Proposition~\ref{prop:V4} predicts. Of the $244$ mixed sum-balanced
	index-four subgroups, $207$ are imprimitive, $28$ have primitive local
	type $(4,2)$ or $(2,4)$, and $9$ have type $(4,4)$. Thus most equalities
	at this index in the stated range are classified by Theorem~\ref{thm:all-imprimitive-quartic}. The primitive cases are less frequent and
	require the partial-sum conditions of Section~5. Unlike the fixed-generator distribution studied by Engberg \cite{Engberg2014}, these are exact counts of eligible cyclic subgroups over a bounded range of prime products, and they do not imply an asymptotic density.

	\subsection{Subgroup balance versus balance of all cosets}
	
	The distinction between a centered sum over $H$ and centered sums
	over \emph{all} cosets can be made exact at index four.
	
	\begin{proposition}[Coset-moment pattern at index four]
		\label{prop:index4-coset-pattern}
		Let $H\le G$ have index four and $-1\notin H$, with
		$H^\perp=\langle\psi\rangle\simeq C_4$. Choose $g\in G$ such that
		$\psi(g)=i$, and put $A=A_{pq}(\psi)$. For $j=0,1,2,3$,
		\begin{equation}\label{eq:index4-coset-pattern}
		\sum_{h\in g^jH}\Res{h}{pq}-\frac{pq|H|}{2}
		=\frac12\Ree\bigl((-i)^j A\bigr).
		\end{equation}
		If $H$ is sum-balanced, then $A=it$ for some real $t$, and the
		four deviations in \eqref{eq:index4-coset-pattern} are
		\[
		\left(0,\frac t2,0,-\frac t2\right).
		\]
		Consequently, such a sum-balanced subgroup of index four
		has every coset sum-balanced if and only if $A_{pq}(\psi)=0$.
	\end{proposition}
	
	\begin{proof}
		Apply character orthogonality to the indicator of $g^jH$.
		The principal character contributes $pq|H|/2$, the nonprincipal
		even character $\psi^2$ has zero first moment, and the two odd
		characters $\psi$ and $\overline\psi$ yield
		\[
		\frac14\left(\overline{\psi(g^j)}A
		+\psi(g^j)\overline A\right)
		=\frac12\Ree\bigl((-i)^jA\bigr).
		\]
		Taking $j=0$ gives $\Ree A=0$ precisely when $H$ is
		sum-balanced. The remaining assertions follow.
	\end{proof}
	
	The refined census in Table~\ref{tab:coset-counts} checks
	all four coset sums directly, independently of the
	character-based test. Among the $244$ mixed sum-balanced index-four
	subgroups, $143$ have every coset centered, while $101$ have
	$A_{pq}(\psi)\ne0$ and hence a noncentered coset. Every mixed sum-balanced subgroup with a primitive quartic annihilator
    belongs to the second class, for every modulus in our setting.
    Indeed, \cite[Lemma~2.2]{PomeranceUlmer} gives a nonzero half-interval
    sum for a primitive odd character, and the relation following
    Theorem~\ref{thm:character-criterion} implies $A_{pq}(\psi)\ne0$.
    Proposition~\ref{prop:index4-coset-pattern} therefore rules out
    centering every coset. In particular, all $37$ primitive balanced
    cases in the census have exactly two centered cosets.
	
    \begin{corollary}[Primitive odd quartic obstruction]
    \label{cor:primitive-odd-quartic-obstruction}
    If $\psi$ is primitive and odd of order four modulo $pq$, then
    $\ker\psi$ is not coset-balanced. If its residue sum is centered,
    exactly two of its four cosets have centered residue sums.
    \end{corollary}
    \begin{proof}
    By Pomerance--Ulmer~\cite[Lemma~2.2]{PomeranceUlmer}, the half-interval
    sum of the primitive odd character $\psi$ is nonzero. The relation
    following Theorem~\ref{thm:character-criterion} has nonzero multiplier
    $\psi(2)^{-1}-2$, so $A_{pq}(\psi)\ne0$. Proposition~\ref{prop:index4-coset-pattern}
    then gives the claim: a centered $H$ has exactly two centered cosets.
    \end{proof}

	\begin{table}[htbp]
		\centering
		\begin{tabular}{l|r|r|r}
			Local quartic type & Every coset centered & Not every coset centered & Total\\ \hline
			Imprimitive $(4,1),(1,4)$ & $143$ & $64$ & $207$\\
			Primitive $(4,2),(2,4)$ & $0$ & $28$ & $28$\\
			Primitive $(4,4)$ & $0$ & $9$ & $9$\\ \hline
			All & $143$ & $101$ & $244$
		\end{tabular}
		\caption{Mixed sum-balanced cyclic subgroups of index four,
			for $p<q$ odd primes and $pq\le10000$, distinguished by whether
			\emph{all} cosets have centered first moment. The second category
			means that $H$ is centered but at least one coset is not.}
		\label{tab:coset-counts}
	\end{table}
	
	\begin{example}[Explicit four-coset discrepancy]\label{ex:coset85}
		Let $H=\langle14\rangle$ modulo $85=5\cdot17$.
		Then $|H|=16$, $[G:H]=4$, and $2H$ generates $G/H$.
		Direct summation over the four cosets gives
		\[
		\begin{array}{c|rrrr}
		\text{Coset}&H&2H&4H&8H\\ \hline
		\text{Sum of residues}&680&765&680&595
		\end{array}
		\]
		Since $85\cdot16/2=680$, the subgroup itself is sum-balanced,
		but $2H$ and $8H$ are not. The deviations $(0,85,0,-85)$
		exhibit the pattern of Proposition~\ref{prop:index4-coset-pattern}.
	\end{example}

\section{Concluding remarks}
Complementary pairing settles the subgroups containing $-1$.
For mixed cyclic subgroups, the embedding in the two local unit groups
controls whether the odd-character moments vanish or cancel.
At index two this reduces to the known quadratic-residuacity criterion;
at index four the biquadratic case has a strict deficit.
Theorem~\ref{thm:all-imprimitive-quartic} classifies every eligible
imprimitive quartic kernel through two local odd integers and gives
the fixed-conductor proportions $1/4$ and $1/2$. The conductor-$13$
case supplies explicit progressions with subgroup balance but without
full coset balance.

For primitive quartic kernels, the triangular specialization and the
short-sum criteria give exact tests, while the classical relative
class-number formula yields the obstruction
$h^-(F)=2b^2$ whenever a coset is centered. These results leave open
a general arithmetic classification of primitive quartic vanishing,
including when
\[
 \Ree\bigl(\chi_p(-q)T_{p,q}(\chi_p,\eps_q)\bigr)=0.
\]
The short sums at $p=5$, the three-term test at $p=13$, and the
complex sums of type $(4,4)$ give specific cases for further study.
For the imprimitive classification, a separate local question is which
prime conductors satisfy $|u|=|v|$.

At index six, the structural enumeration identifies the possible cyclic
kernels before their sums are tested. Theorem~\ref{thm:sextic19family}
gives infinitely many imprimitive examples in which nonzero quadratic
and sextic contributions cancel, with four of the six cosets centered.
Infinitude for primitive sextic annihilators remains open here.
Finally, the census distinguishes the $143$ mixed balanced index-four
subgroups with every coset centered from the $101$ with exactly two
centered cosets, within $pq\le10000$. Its complete data and exact
verification programs accompany the manuscript.

\section*{Statements and Declarations}
\subsection*{Funding}
The authors received no funding for this work.

\subsection*{Competing interests}
The authors declare no competing interests.

\subsection*{Data availability}
The data and code supporting the computational results are supplied in
\nolinkurl{data-package-v12/} and the accompanying
\nolinkurl{residue-sums-pq-data-v12.zip}; the canonical public repository
is \url{https://github.com/the-omega-institute/residue-sums-pq-data}.
The package contains the complete subgroup census, verification scripts,
exact outputs, the formal-source provenance record, reproduction
instructions, and SHA-256 checksums. The data are licensed under CC BY 4.0
and the code under the MIT licence.

	\subsection*{Author contributions}
	
	H.~Ma and W.~Zhang initiated the collaboration by providing numerical
	computations that revealed counterexamples to M.~Merca's original
	conjecture and by suggesting a systematic investigation of the cases
	in which the conjectured equality remains valid.
	Building on these findings, M.~Merca obtained initial classification
	results, proposed the first structure of the manuscript, and prepared
	its first draft. Subsequent mathematical developments and revisions
	involved contributions from all three authors.
    H.~Ma and W.~Zhang also contributed the Lean~4 formalizations
    of the introductory counterexample and lifted prime-modulus identity, and the exact census and
    verification code.

	\subsection*{Use of AI tools}
	The contributions of H.~Ma and W.~Zhang were AI-assisted: OpenAI's ChatGPT
	and Codex and Anthropic's Claude were used for literature search,
	computations and checking proofs. All proofs in the paper can be checked
	by hand, and the computations can be reproduced from the data package.
	H.~Ma and W.~Zhang take full responsibility for their contributions.
	
	\par\smallskip
    \noindent M.~Merca used ChatGPT (OpenAI) to assist with English-language proofreading, stylistic editing, 	and improving the clarity of the mathematical 	exposition. All AI-assisted editorial revisions were carefully reviewed by the author, who 	retains full responsibility for his contributions.

\end{document}